\documentclass[12pt]{article}%
\usepackage{amsmath}
\usepackage{amsfonts}
\usepackage{amssymb}
\usepackage{graphicx}
\usepackage{graphics}%
\providecommand{\U}[1]{\protect\rule{.1in}{.1in}}
\newtheorem{theorem}{Theorem}

\newtheorem{result}[theorem]{Result}

\newenvironment{proof}[1][Proof]{\noindent\textbf{#1.} }{\ \rule{0.5em}{0.5em}}
\begin{document}

\title{Asymptotic analysis of a fluid model modulated by an $M/M/1$ queue}
\author{Charles Knessl\\Department of Mathematics, Statistics and Computer Science \\University of Illinois at Chicago (M/C 249) \\851 South Morgan Street \\Chicago, IL 60607-7045\\knessl@uic.edu
\and Diego Dominici\\Department of Mathematics\\State University of New York at New Paltz\\1 Hawk Dr. Suite 9\\New Paltz, NY 12561-2443\\dominicd@newpaltz.edu}
\maketitle

\begin{abstract}
We analyze asymptotically a differential-difference equation, that arises in a
Markov-modulated fluid model. We use singular perturbation methods to analyze
the problem with appropriate scalings of the two
state variables. In particular, the ray method and asymptotic matching are used.

\end{abstract}

Keywords: Fluid models, M/M/1 queue, differential-difference equations, ray method, asymptotics.

\section{Introduction}

Fluid models have received much recent attention in the literature. They have
been used to model statistical multiplexers in ATM (asynchronous transfer
mode) networks \cite{Ren}, \cite{MR95j:60151}, \cite{tanaka}, packet speech
multiplexers \cite{tucker}, buffer storage in manufacturing models
\cite{wijngaard}, buffer memory in store-and-forward systems \cite{hashida}
and high-speed digital communication networks \cite{MR1295412}. In these
models the queue length is considered a continuous (or \textquotedblleft
fluid\textquotedblright) process, rather than a discrete random process that
measures the number of customers. These models tend to be somewhat easier to
analyze, as they allow for less randomness than more traditional queueing models.

The following is a description of a fairly general fluid model, of which many
variants and special cases have been considered. Let $X(t)$ denote the amount
of fluid at time $t$ in the buffer. Furthermore, let $Z(t)$ be a
continuous-time Markov process. The content of the buffer $X(t)$ is regulated
(or driven) by $Z(t)$ in such a way that the \textit{net input rate} into the
buffer (i.e., the rate of change of its content) is $\eta\left[  Z(t)\right]
$. The function $\eta\left(  \cdot\right)  $ is called the \textit{drift
function}. When the buffer capacity is infinite, the dynamics of $X(t)$ are
given by%
\begin{equation}
\frac{dX}{dt}=\left\{
\begin{array}
[c]{c}%
\eta\left[  Z(t)\right]  ,\quad X(t)>0\\
\max\left\{  \eta\left[  Z(t)\right]  ,0\right\}  ,\quad X(t)=0
\end{array}
\right.  . \label{fluid infinity}%
\end{equation}
The condition at $X(t)=0$ ensures that the process $X(t)$ does not become
negative. When the buffer capacity is finite, say $B,$ the dynamics are given
by%
\[
\frac{dX}{dt}=\left\{
\begin{array}
[c]{c}%
\eta\left[  Z(t)\right]  ,\quad0<X(t)<B\\
\max\left\{  \eta\left[  Z(t)\right]  ,0\right\}  ,\quad X(t)=0\\
\min\left\{  \eta\left[  Z(t)\right]  ,0\right\}  ,\quad X(t)=B
\end{array}
\right.  .
\]
The condition at $X(t)=B$ prevents the buffer content from exceeding $B.$

In many applications, the process $Z(t)$ evolves as a finite or infinite state
\textit{birth-death process. }The description of the motion of $Z(t)$ is as
follows: the process sojourns in a given state $k$ for a random length of
time, whose distribution is exponential with parameter $\lambda_{k}+\mu_{k}.$
When leaving state $k,$ the process enters either state $k+1$ or state $k-1$
with probabilities%

\[
k\rightarrow k+1\quad w.p.\quad\frac{\lambda_{k}}{\lambda_{k}+\mu_{k}},\quad
k\in\mathcal{N}%
\]

\[
k\rightarrow k-1\quad w.p.\quad\frac{\mu_{k}}{\lambda_{k}+\mu_{k}},\quad
k\in\mathcal{N}.
\]

The motion is analogous to that of a random walk, except that transitions
occur at random rather than fixed times. The parameters $\lambda_{k}$ and
$\mu_{k}$ are called, respectively, the \textit{birth }and \textit{death
}rates. We shall assume that the birth and death rates are positive with the
exception of the death rate $\mu_{0}$ in the lowest state and (in case of a
finite space ${\mathcal{N}}=\{0,1,\ldots,N \})$ the birth rate $\lambda_{N}$
in the highest state, which are equal to zero. Also, it will be convenient to
interpret $\lambda_{k}$ and $\mu_{k}$ as zero if $k\notin\mathcal{N}.$

If the buffer has emptied at time $t,$ it remains empty as long as the drift
is negative. We let $\eta\left[  Z\left(  t\right)  \right]  =r_{k},$ given
that $Z(t)$ is in state $k.$ We shall assume throughout that $r_{k}\neq0$ for
all states. We shall also assume that $r_{k}>0$ for at least one
$k\in\mathcal{N}$, since otherwise, in the steady state, the buffer is always empty.

We let%
\[
\pi_{k}=%
{\displaystyle\prod\limits_{j=0}^{k-1}}
\frac{\lambda_{j}}{\mu_{j+1}},\quad k\in\mathcal{N}%
\]
where an empty product should be interpreted as unity. The stationary
probabilities $p_{k}$ of the birth-death process can then be represented as%
\[
p_{k}=\frac{\pi_{k}}{%
{\displaystyle\sum\limits_{j\in\mathcal{N}}}
\pi_{j}},\quad k\in\mathcal{N}.
\]

When the capacity of the buffer is infinitely large, in order that a
stationary distribution for $X(t)$ exists, the mean drift $%
{\displaystyle\sum\nolimits_{k\in\mathcal{N}}}
p_{k}r_{k}$ should be negative or, equivalently, the following
\textit{stability condition }should be satisfied
\begin{equation}%
{\displaystyle\sum\limits_{k\in\mathcal{N}}}
\pi_{k}r_{k}<0. \label{2.3}%
\end{equation}
We let
\begin{align*}
{\mathcal{N}}^{\ +}  &  =\left\{  k\in{\mathcal{N}}\mid r_{k}>0\right\}
,\quad{\mathcal{N}}^{\ -}=\left\{  k\in{\mathcal{N}}\mid r_{k}<0\right\}  ,\\
N_{+}  &  =\left\vert \mathcal{N}^{\ +}\right\vert ,\quad N_{-}=\left\vert
\mathcal{N}^{\ -}\right\vert
\end{align*}
and since we assume that the drift in each state is nonzero, we have
$\mathcal{N}^{\ +}\cup\mathcal{N}^{\ -}=\mathcal{N}.$

Setting
\[
P_{k}(t,x)=\Pr\left[  X(t)\leq x,\ Z(t)=k\right]  ;\quad t,\ x\geq0,\quad
k\in\mathcal{N,}%
\]
the Kolmogorov forward equations for the Markov process $\left[
X(t),Z(t)\right]  $ are given by%
\[
\frac{\partial P_{k}}{\partial t}+r_{k}\frac{\partial P_{k}}{\partial
x}=\lambda_{k-1}P_{k-1}+\mu_{k+1}P_{k+1}-\left(  \lambda_{k}+\mu_{k}\right)
P_{k},\quad k\in\mathcal{N}.
\]
For the stationary distribution
\[
F_{k}(x)\equiv{\lim}_{t\rightarrow\infty}P_{k}(t,x)
\]
we have%
\begin{equation}
r_{k}F_{k}^{\prime}=\lambda_{k-1}F_{k-1}+\mu_{k+1}F_{k+1}-\left(  \lambda
_{k}+\mu_{k}\right)  F_{k},\quad k\in\mathcal{N}. \label{2.7}%
\end{equation}
Since the buffer content is increasing whenever the drift is positive, the
solution to (\ref{2.7}) must satisfy the boundary conditions%
\begin{equation}
F_{k}(0)=0,\quad k\in\mathcal{N}^{\ +}. \label{2.8}%
\end{equation}
This means that there is no probability mass at $x=0$ if the drift takes you
away from the boundary. Also, we must have%
\begin{equation}
F_{k}(\infty)=p_{k},\quad k\in\mathcal{N}, \label{2.9}%
\end{equation}
as this is the marginal distribution of the regulating process $Z(t)$. In the
finite capacity case we have the additional boundary condition%
\begin{equation}
F_{k}(B)=p_{k},\quad k\in\mathcal{N}^{\ -}. \label{2.10}%
\end{equation}
This means there is no probability mass at $x=B$ if the drift moves the
process from this boundary. The values of $F_{k}(0)$ for $k\in\mathcal{N}%
^{\ -}$, and of $F_{k}(B)$ for $k\in\mathcal{N}^{\ +}$, are not a priori
known. The \textquotedblleft half" boundary conditions (\ref{2.8}) and
(\ref{2.10}) make these problems difficult.

The purpose of this paper is to continue our asymptotic analysis of fluid
models using the ray method \cite{MR80g:35002}, which we successfully applied
in \cite{MR2117327} to the model first studied by Anick, Mitra and Sondhi in
\cite{MR84a:68020}.

The paper is organized as follows. In Section 2 we state the basic equations.
In Sections 3-7 we analyze these in various ranges of the state space
(\ref{SS}). In Section 8\ we study the marginal distribution. Finally, in
Section 9 we summarize and interpret the results.

\section{Problem statement}

Let ${\mathcal{N=}}\left\{  0,1,2,\ldots\right\}  $, and the parameters
$\lambda_{k}$ and $\mu_{k}$ be constant,
\[
\lambda_{k}=\lambda,\quad\mu_{k}=\left\{
\begin{array}
[c]{c}%
\mu,\quad1\leq k\\
0,\quad k=0
\end{array}
\right.  ,\quad\quad\rho=\frac{\lambda}{\mu}<1.
\]
The drift is taken as $r_{k}=k-c,$ where $c$ represents the output rate of the
buffer. We assume $c$ to be a positive non-integer number. This model
corresponds to a fluid model modulated by the standard $M/M/1$ queue.

The forward Kolmogorov equations for $F_{k}(x)$ are then%
\begin{equation}
(k-c)F_{k}^{\prime}(x)=\lambda F_{k-1}(x)+\mu F_{k+1}(x)-\left(  \lambda
+\mu\right)  F_{k}(x),\quad0\leq k \label{diffeq}%
\end{equation}%
\begin{equation}
\mu F_{0}(x)=\lambda F_{-1}(x). \label{BC0}%
\end{equation}
with boundary conditions%
\begin{equation}
F_{k}(0)=0,\quad\left\lfloor c\right\rfloor +1\leq k. \label{FBC}%
\end{equation}
and limiting distribution%
\begin{equation}
F_{k}\left(  \infty\right)  =\left(  1-\rho\right)  \rho^{k},\quad0\leq k.
\label{Finf}%
\end{equation}
Here (\ref{BC0}) defines $F_{-1}(x)$ and this condition is equivalent to%
\[
-cF_{0}^{\prime}(x)=\mu F_{1}(x)-\lambda F_{0}(x).
\]
Since the buffer capacity is infinite, we need the stability condition%
\begin{equation}
\frac{\rho}{1-\rho}<c,\quad\text{or\quad}\rho<1-\frac{1}{c+1}.
\label{stability}%
\end{equation}

A related model, with $r_{0}=\rho_{0}<0$ and $r_{k}=\rho>0$ was studied in
\cite{MR95a:60138}, \cite{MR2003c:60147} and \cite{MR97j:60165} where a
spectral representation of the solution was obtained. The same model was
analyzed in \cite{MR2003i:60165} using continued fractions. The general case,
with arbitrary $r_{k},\mu_{k}$ and $\lambda_{k}$, was studied in
\cite{MR2002h:60210} and \cite{MR2041913} using a family of orthogonal
polynomials. The fluid queue driven by a general Markovian process was
analyzed in \cite{MR2002f:60182}. A numerical method was presented in
\cite{MR99i:60170}. The full transient solution was considered in
\cite{MR2156569}.

We shall analyze this model directly by using the differential-difference
equation (\ref{diffeq}) satisfied by $F_{k}(x)$. After appropriate scalings of
$k$ and $x,$ we analyze this equation asymptotically for $c\rightarrow\infty,$
using singular perturbation methods. We also carefully treat various boundary
and corner regions of the state space
\begin{equation}
\left\{  \left(  x,k\right)  :x\geq0,\quad0\leq k\right\}  , \label{SS}%
\end{equation}
and indeed we show that their analysis is needed in order to obtain the
asymptotic expansions away from the boundaries.

\section{The ray expansion}

To analyze the problem (\ref{diffeq})-(\ref{Finf}) for large \ $c$ \ we
introduce the scaled variables $y$ and $z,$ with
\[
x=c^{2}y,\quad k=cz,\quad z,y=O(1).
\]
We define the function \ $G(y,z)$ \ and the small parameter \ $\varepsilon$
\ by \ \
\[
\varepsilon=c^{-1},\quad F_{k}(x)=G\left(  x\varepsilon^{2},k\varepsilon
\right)  =G(y,z)
\]
and note that \ $F_{k\pm1}(x)=G(y,z\pm\varepsilon).$

Then (\ref{diffeq}) becomes the following equation for\ $G(y,z)$\
\begin{equation}
\varepsilon(z-1)\frac{\partial G}{\partial y}(y,z)=\lambda G(y,z-\varepsilon
)+\mu G(y,z+\varepsilon)-\left(  \lambda+\mu\right)  G(y,z) \label{eqG}%
\end{equation}
and (\ref{FBC}) implies that%
\begin{equation}
G(0,z)=0,\quad1<z. \label{BC}%
\end{equation}
Also, from (\ref{Finf}), we have
\begin{equation}
F_{k}(\infty)=G(\infty,z)=\left(  1-\rho\right)  \exp\left[  \frac
{1}{\varepsilon}z\ln\left(  \rho\right)  \right]  ,\quad0<z. \label{Ginfinity}%
\end{equation}

To find \ $G(y,z)$ for\ \ $\varepsilon$ \ small, we shall use the ray method.
Thus, we consider solutions which have the asymptotic form
\begin{equation}
G(y,z)\sim\varepsilon^{\nu}\exp\left[  \frac{1}{\varepsilon}\Psi(y,z)\right]
\mathbb{K}(y,z), \label{GRay}%
\end{equation}
where $\nu$ is a constant that must be determined. Using\ (\ref{GRay}) in
(\ref{eqG}), with
\[
\frac{1}{\varepsilon}\Psi(y,z\pm\varepsilon)=\frac{1}{\varepsilon}\Psi\pm
\Psi_{z}+\frac{1}{2}\Psi_{zz}\varepsilon+O\left(  \varepsilon^{2}\right)  ,
\]
dividing by \ $\exp\left[  \frac{1}{\varepsilon}\Psi(y,z)\right]  ,$ and
expanding in powers of \ $\varepsilon$ \ we obtain the \emph{eikonal equation}
for \ $\Psi(y,z)$%
\begin{equation}
\mu\left(  1-e^{q}\right)  +\lambda\left(  1-e^{-q}\right)  +\left(
z-1\right)  p=0, \label{eik}%
\end{equation}
and the \emph{transport equation} for $\mathbb{K}(y,z)$%
\begin{equation}
\left(  \mu e^{q}-\lambda e^{-q}\right)  \frac{\partial\mathbb{K}}{\partial
z}+\left(  1-z\right)  \frac{\partial\mathbb{K}}{\partial y}+\frac{1}{2}%
\frac{\partial q}{\partial z}\left(  \mu e^{q}+\lambda e^{-q}\right)
\mathbb{K}=0, \label{trans}%
\end{equation}
where%
\[
p=\frac{\partial\Psi}{\partial y},\quad q=\frac{\partial\Psi}{\partial z}.
\]
To solve (\ref{eik}) and (\ref{trans}) we use the method of characteristics,
which we briefly review below.

Given the first order partial differential equation%
\[
\mathfrak{F}\left(  y,z,\Psi,p,q\right)  =0,
\]
where \ $p=\Psi_{y},\quad q=\Psi_{z},$ we search for a solution \ $\Psi(y,z).$
The technique is to solve the system of \textquotedblleft characteristic
equations\textquotedblright\ given by%
\begin{align*}
\dot{y}  &  =\frac{\partial y}{\partial t}=\mathfrak{F}_{p},\quad\dot
{z}=\mathfrak{F}_{q}\\
\dot{p}  &  =-\mathfrak{F}_{y}-p\mathfrak{F}_{\Psi},\quad\dot{q}%
=-\mathfrak{F}_{z}-q\mathfrak{F}_{\Psi}\\
\dot{\psi}  &  =p\mathfrak{F}_{p}+q\mathfrak{F}_{q}%
\end{align*}
where we now consider $\left\{  y,z,\psi,p,q\right\}  $ to all be functions of
the variables $s$ and $t,$ with $\psi(s,t)=\Psi(y,z).$ Here $t$ measures how
far we are along a particular characteristic curve or ray and $s$ indexes them.

For the eikonal equation (\ref{eik}), the characteristic equations are
\begin{subequations}
\label{strip}%
\begin{align}
\dot{y}  &  =z-1\label{eqa}\\
\quad\dot{z}  &  =\lambda e^{-q}-\mu e^{q}\label{eqb}\\
\dot{p}  &  =0\label{eqc}\\
\quad\dot{q}  &  =-p\label{eqd}\\
\dot{\psi}  &  =p(z-1)+q\left(  \lambda e^{-q}-\mu e^{q}\right)  . \label{eqe}%
\end{align}
The particular solution is determined by the initial conditions at $t=0$. We
shall show that for this problem two different types of solutions are needed;
these correspond to two distinct families of rays.

Setting \ $\left.  \Psi_{y}\right\vert _{t=0}=s,\ \left.  \Psi_{z}\right\vert
_{t=0}=B$ \ and solving (\ref{eqc})-(\ref{eqd}), yields%
\end{subequations}
\begin{equation}
p=s,\quad q=B-st \label{pq}%
\end{equation}
so that $\Psi_{y}$ is constant along a ray.

\subsection{The rays from $(0,1)$}

We now consider the family of rays emanating from the point \ $y=0,\ z=1.$
\ Evaluating (\ref{eik}) at $t=0$ we get%
\[
\mu\left(  1-e^{B}\right)  +\lambda\left(  1-e^{-B}\right)  =0
\]
so that%
\begin{equation}
B=0\ \quad\text{or}\ \quad B=\ln\left(  \rho\right)  . \label{u0}%
\end{equation}
From (\ref{eqb}) and (\ref{pq}), with the initial condition \ $z(s,0)=1$ and
using (\ref{u0}), \ we obtain
\begin{equation}
z=\frac{1}{s}\left[  \lambda e^{-B}\left(  e^{st}-1\right)  +\mu e^{B}\left(
e^{-st}-1\right)  \right]  +1. \label{z(u)}%
\end{equation}

From (\ref{eqa}), we have \
\[
\dot{y}(s,0)=z(s,0)-1=0
\]
and%
\[
\ddot{y}(s,0)=\dot{z}(s,0)=\lambda e^{-B}-\mu e^{B}.
\]
From (\ref{u0}) we have%
\[
\ddot{y}(s,0)=\left\{
\begin{array}
[c]{c}%
\lambda-\mu<0,\quad B=0\\
\mu-\lambda>0,\quad B=\ln\left(  \rho\right)
\end{array}
\right.  .
\]
Using the initial condition \ $y(s,0)=0$ \ and expanding in powers of \ $t$,
\ we get%
\[
y(s,t)\sim\ddot{y}(s,0)\frac{t^{2}}{2},\quad t\rightarrow0
\]
and in order to have \ $y>0$ \ for $t>0$ (i.e., for the rays to enter the
domain $[0,\infty)\times\lbrack0,\infty)$) we need to choose
\begin{equation}
B=\ln\left(  \rho\right)  \label{U0}%
\end{equation}
with $\ B<0$ since $\rho<1$.

Integrating (\ref{eqa}) and using (\ref{z(u)}) and (\ref{U0}), we conclude
that%
\begin{equation}
y(s,t)=\frac{1}{s^{2}}\left[  \mu\left(  e^{st}-st-1\right)  +\lambda\left(
1-st-e^{-st}\right)  \right]  \label{y(s,t)}%
\end{equation}%
\begin{equation}
z(s,t)=\frac{1}{s}\left[  \mu\left(  e^{st}-1\right)  +\lambda\left(
e^{-st}-1\right)  \right]  +1. \label{z(s,t)}%
\end{equation}
This yields the rays that emanate from $(0,1)$ in parametric form. Several
rays are sketched in Figure \ref{raycor}.

\begin{figure}[t]
\begin{center}
\rotatebox{270} {\resizebox{10cm}{!}{\includegraphics{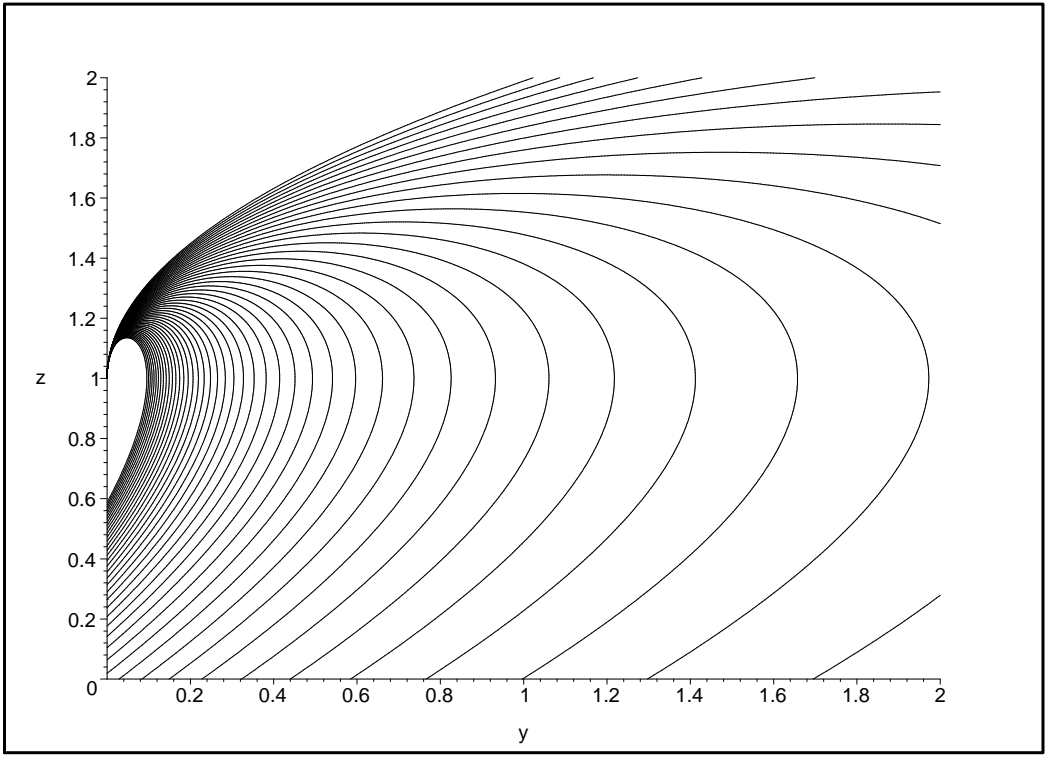}}}
\end{center}
\caption{A sketch of the rays from $(0,1)$.}%
\label{raycor}%
\end{figure}

For \ $t\geq0$ \ and each value of \ $s$, \ (\ref{y(s,t)}) and (\ref{z(s,t)})
determine a ray in the\ $(y,z)$ plane, which starts from \ $(0,1)$ \ at
\ $t=0$. \ We discuss a particular ray which can be obtained in an explicit
form. For \ $s=0$ \ we can eliminate $\ t$ \ from (\ref{z(s,t)}) and obtain%
\begin{equation}
y=Y_{0}(z):=\frac{\left(  z-1\right)  ^{2}}{2\left(  \mu-\lambda\right)
},\quad s=0,\quad1\leq z, \label{Y0}%
\end{equation}
and along this ray, $t$ and $z$ are related by%
\begin{equation}
t\left(  Y_{0},z\right)  =T_{0}(z)=\frac{z-1}{\mu-\lambda},\quad
s=0,\quad1\leq z. \label{T0}%
\end{equation}

For \ $s>0$, \ we have both \ $y(s,t)$ \ and \ $z(s,t)$ \ increasing for
\ $t>0$. \ For \ $s<0$ \ the rays reach a maximum value in \ $z$ \ at
\ $t=T_{1}$, where%

\[
T_{1}(s)=\frac{1}{2s}\ln\left(  \rho\right)  ,\quad s<0
\]
and we have%
\begin{equation}
y(s,T_{1})=\frac{1}{s^{2}}\left[  \lambda-\mu-\frac{1}{2}\left(  \lambda
+\mu\right)  \ln\left(  \rho\right)  \right]  \label{ytmax}%
\end{equation}%
\begin{equation}
z(s,T_{1})=\frac{1}{s}\left[  2\sqrt{\lambda\mu}-\left(  \lambda+\mu\right)
\right]  +1. \label{ztmax}%
\end{equation}

\ From (\ref{eqa}) we see that the maximum value in $y$ is achieved at the
same time that \ $z=1$, and that occurs at \ $t=T_{2}$ \ with%
\begin{equation}
T_{2}(s)=\frac{1}{s}\ln\left(  \rho\right)  ,\quad s<0 \label{Tg}%
\end{equation}
and
\[
y(s,T_{2})=\frac{1}{s^{2}}\left[  2\left(  \lambda-\mu\right)  -\left(
\lambda+\mu\right)  \ln\left(  \rho\right)  \right]  .
\]

Inverting the equations (\ref{y(s,t)})-(\ref{z(s,t)}) we can write%
\[
s=S\left(  y,z\right)  ,\quad t=T(y,z)
\]
and%
\[
\Psi(y,z)=\psi\left[  S\left(  y,z\right)  ,T(y,z)\right]  ,\quad
\mathbb{K(}y,z\mathbb{)}=K\left[  S\left(  y,z\right)  ,T(y,z)\right]  .
\]
We will use this notation in the rest of the article.

\subsection{The functions $\Psi$ and $\mathbb{K}$}

From (\ref{eqe}) we have%
\[
\dot{\psi}=\mu e^{st}\left[  1+\ln\left(  \rho\right)  -ts\right]  +\lambda
e^{-st}\left[  1-\ln\left(  \rho\right)  +ts\right]  ,
\]
which we can integrate to get%
\begin{align}
\psi(s,t)  &  =\frac{\mu}{s}e^{st}\left[  2+\ln\left(  \rho\right)
-ts\right]  -\frac{\lambda}{s}e^{-st}\left[  2-\ln\left(  \rho\right)
+ts\right] \label{psi1}\\
&  +\psi(s,0)-\frac{\mu}{s}\left[  2+\ln\left(  \rho\right)  \right]
+\frac{\lambda}{s}\left[  2-\ln\left(  \rho\right)  \right]  .\nonumber
\end{align}
Obviously, $\psi(s,0)\equiv\psi_{0}$ is a constant, since all rays start at
the same point. Setting $s=0$ in (\ref{psi1}) and using (\ref{T0}), we obtain%
\[
\psi(0,t)=\psi_{0}+\left(  z-1\right)  \ln\left(  \rho\right)
\]
and therefore, taking the limit as $t\rightarrow\infty,$ we get%
\[
\Psi(\infty,z)=\psi_{0}+\left(  z-1\right)  \ln\left(  \rho\right)  .
\]
On the other hand, from (\ref{Ginfinity}) we have%
\[
\Psi(\infty,z)=z\ln\left(  \rho\right)
\]
and we conclude that%
\[
\psi_{0}=\ln\left(  \rho\right)  .
\]
Solving for $e^{st}$ in (\ref{y(s,t)})-(\ref{z(s,t)}), we get%
\begin{equation}
e^{st}=1+\frac{s}{2\mu}\left[  z-1+ys+t\left(  \lambda+\mu\right)  \right]
,\quad e^{-st}=1+\frac{s}{2\lambda}\left[  z-1-ys-t\left(  \lambda+\mu\right)
\right]  . \label{expA}%
\end{equation}
Replacing (\ref{expA}) in (\ref{psi1}), we obtain%
\begin{equation}
\psi=2ys+\left[  \ln\left(  \rho\right)  -st\right]  \left(  z-1\right)
+\ln\left(  \rho\right)  . \label{psi2}%
\end{equation}

We shall now solve the transport equation (\ref{trans}), which we rewrite as%
\begin{equation}
\left(  z-1\right)  \frac{\partial\mathbb{K}}{\partial y}+\left(  \lambda
e^{-q}-\mu e^{q}\right)  \frac{\partial\mathbb{K}}{\partial z}=\frac{1}%
{2}\frac{\partial q}{\partial z}\left(  \mu e^{q}+\lambda e^{-q}\right)
\mathbb{K}. \label{transp1}%
\end{equation}
Using (\ref{y(s,t)}) and (\ref{z(s,t)}) in (\ref{transp1}), we have%
\begin{equation}
\frac{\partial K}{\partial t}=\frac{1}{2}\frac{\partial q}{\partial z}\left(
\mu e^{q}+\lambda e^{-q}\right)  K. \label{transp2}%
\end{equation}
To solve (\ref{transp2}), we need to compute $\frac{\partial q}{\partial z}$
as a function of $s$ and $t.$ Use of the chain rule gives%
\[%
\begin{bmatrix}
\frac{\partial y}{\partial t} & \frac{\partial y}{\partial s}\\
\frac{\partial z}{\partial t} & \frac{\partial z}{\partial s}%
\end{bmatrix}%
\begin{bmatrix}
\frac{\partial t}{\partial y} & \frac{\partial t}{\partial z}\\
\frac{\partial s}{\partial y} & \frac{\partial s}{\partial z}%
\end{bmatrix}
=%
\begin{bmatrix}
1 & 0\\
0 & 1
\end{bmatrix}
\]
and hence,%
\begin{equation}%
\begin{bmatrix}
\frac{\partial t}{\partial y} & \frac{\partial t}{\partial z}\\
\frac{\partial s}{\partial y} & \frac{\partial s}{\partial z}%
\end{bmatrix}
=\frac{1}{\mathbf{J}}%
\begin{bmatrix}
\frac{\partial z}{\partial s} & -\frac{\partial y}{\partial s}\\
-\frac{\partial z}{\partial t} & \frac{\partial y}{\partial t}%
\end{bmatrix}
, \label{inversion}%
\end{equation}
where the Jacobian $\mathbf{J}(s,t)$ is defined by
\begin{equation}
\mathbf{J}(s,t)=\frac{\partial y}{\partial t}\frac{\partial z}{\partial
s}-\frac{\partial y}{\partial s}\frac{\partial z}{\partial t}=\frac{1}%
{s}\left[  2\frac{\partial z}{\partial t}y-\left(  z-1\right)  ^{2}\right]  .
\label{J}%
\end{equation}
Using (\ref{inversion}) we can show after some algebra that%
\begin{equation}
\frac{\partial q}{\partial z}=-\frac{2y}{\mathbf{J}}, \label{qz}%
\end{equation}
while (\ref{expA}) gives%
\[
\mu e^{q}+\lambda e^{-q}=\mu+\lambda+(z-1)s.
\]
Thus, the transport equation (\ref{transp2}) becomes
\begin{equation}
\frac{1}{K}\frac{\partial K}{\partial t}=-\frac{y}{\mathbf{J}}\left[
\mu+\lambda+(z-1)s\right]  . \label{Transp3}%
\end{equation}
Using (\ref{y(s,t)}) and (\ref{z(s,t)}) in (\ref{J}), we have%
\begin{equation}
\frac{\partial\mathbf{J}}{\partial t}=\frac{2y}{s}\frac{\partial^{2}%
z}{\partial t^{2}}=2y\left[  \mu+\lambda+(z-1)s\right]  . \label{transp4}%
\end{equation}
Combining (\ref{Transp3}) and (\ref{transp4}), we obtain%
\[
\frac{1}{K}\frac{\partial K}{\partial t}=-\frac{1}{2\mathbf{J}}\frac
{\partial\mathbf{J}}{\partial t},
\]
whose solution is%
\begin{equation}
K(s,t)=\frac{K_{0}(s)}{\sqrt{\mathbf{J(}s,t\mathbf{)}}}, \label{K}%
\end{equation}
where $K_{0}(s)$ is a function to be determined.

From (\ref{J}), we have%
\begin{equation}
\mathbf{J}(s,t)=\left(  \mu^{2}-\lambda^{2}\right)  \frac{t^{3}}{3}+O\left(
t^{4}\right)  ,\quad t\rightarrow0. \label{J2}%
\end{equation}
Since the Jacobian vanishes as $t\rightarrow0,$ the ray expansion eases to be
valid near the point $(0,1),$ where a separate analysis is needed.

So far we have determined the exponent \ $\psi(s,t)$ \ and the leading
amplitude \ $K(s,t)$ \ except for the function \ $K_{0}(s)$ in (\ref{K})\ and
the power \ $\nu$ \ in (\ref{GRay}). \ In Section\ 4 we will determine them by
matching (\ref{GRay}) to a corner layer solution valid in a neighborhood of
the point \ $(0,1)$.

\subsection{The rays from infinity}

Denoting the domain in the $(y,z)$ plane by%
\begin{equation}
\mathfrak{D=}[0,\infty)\times\lbrack0,\infty), \label{D}%
\end{equation}
we must determine what part of $\mathfrak{D}$ the rays from infinity fill. The
expansion corresponding to these rays must satisfy the boundary condition
(\ref{Ginfinity}). Thus, we have%
\begin{equation}
p(\infty,z)=p_{\infty}=0,\quad q(\infty,z)=q_{\infty}=\ln\left(  \rho\right)
, \label{uinf}%
\end{equation}
while (\ref{eqa})-(\ref{eqb}) yield equations for the rays $y_{\infty
}(t),\ z_{\infty}(t),\ $
\begin{equation}
\dot{y}_{\infty}=z_{\infty}-1,\quad\dot{z}_{\infty}=\mu-\lambda\label{yzinf}%
\end{equation}
or, eliminating $t$ from the system (\ref{yzinf}) and writing $y_{\infty
}(t)=Y_{\infty}(z)$ we get%
\begin{equation}
\frac{dY_{\infty}}{dz}=\frac{z-1}{\mu-\lambda}. \label{Y(z)}%
\end{equation}
Solving (\ref{yzinf}) subject to the initial condition $Y_{\infty}%
(z_{0})=y_{0},$ where $y_{0}\times z_{0}=0,$ we get%
\begin{equation}
Y_{\infty}(z)=y_{0}+\frac{1}{2\left(  \mu-\lambda\right)  }\left[  \left(
z-1\right)  ^{2}-\left(  z_{0}-1\right)  ^{2}\right]  . \label{Yinf}%
\end{equation}

\begin{figure}[t]
\begin{center}
\rotatebox{270} {\resizebox{10cm}{!}{\includegraphics{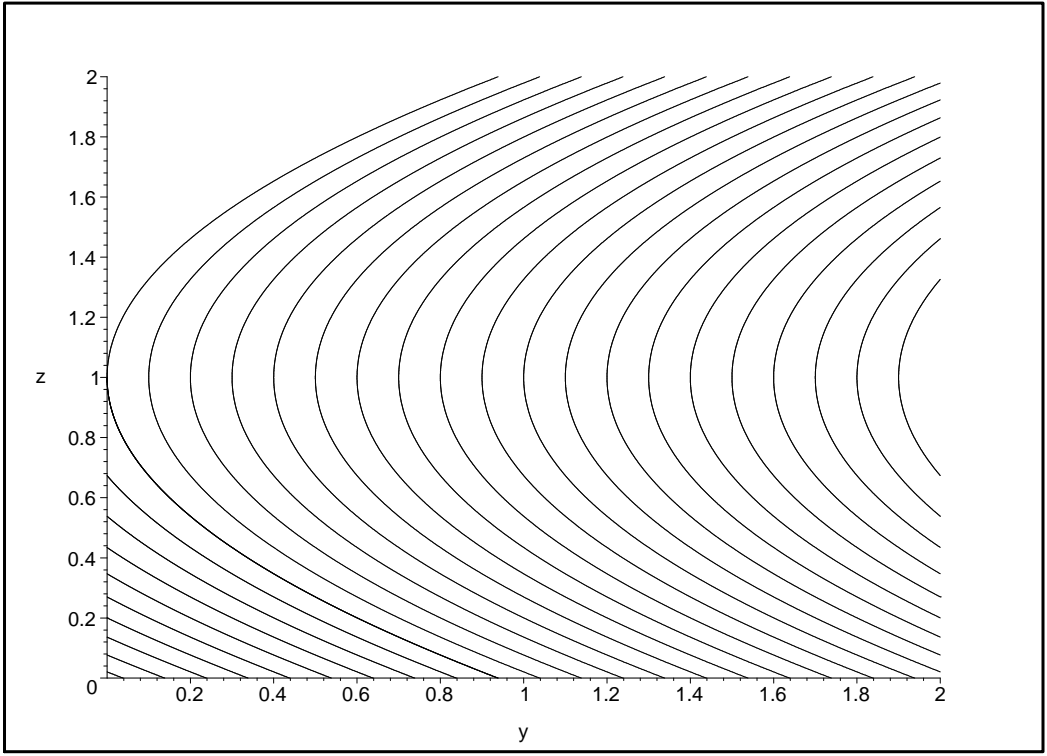}}}
\end{center}
\caption{A sketch of the rays from infinity.}%
\label{rayinf}%
\end{figure}

From (\ref{Y(z)}), it follows that the minimum value in $y$ occurs when $z=1.$
Hence, for $y$ to be positive, we must have
\[
y_{0}>\frac{\left(  z_{0}-1\right)  ^{2}}{2\left(  \mu-\lambda\right)  }%
=Y_{0}(z_{0}),
\]
where $Y_{0}(z)$ was defined in (\ref{Y0}). Therefore, the rays from infinity
fill the region given by%
\begin{equation}
R=\left\{  0\leq y,\quad0\leq z\leq1\right\}  \cup\left\{  Y_{0}(z)\leq
y,\quad1\leq z\right\}  . \label{R}%
\end{equation}
The complementary region \ $R^{C}$%
\begin{equation}
R^{C}=\left\{  0\leq y<Y_{0}(z),\quad1\leq z\right\}  , \label{RC}%
\end{equation}
is a \textit{shadow }of the rays from infinity. In $R^{C}$, \ $G$ \ is given
by (\ref{GRay}) as only the rays from $(0,1)$ are present (see Figure
\ref{rayinf}). In the region $R,$ both the rays coming from\ $(0,1)$ and the
rays coming from infinity must be taken into account. We add (\ref{GRay}) and
(\ref{Ginfinity}) to represent \ $G$ \ in the asymptotic form \
\begin{equation}
G(y,z)\sim(1-\rho)\exp\left[  \frac{1}{\varepsilon}z\ln\left(  \rho\right)
\right]  +\varepsilon^{\nu}\exp\left[  \frac{1}{\varepsilon}\Psi(y,z)\right]
\mathbb{K}(y,z),\quad(y,z)\in R. \label{GR}%
\end{equation}
We can show that $z\ln\left(  \rho\right)  >\Psi(y,z)$ in the interior of $R,$
so that $G(y,z)\sim G(\infty,z).$ However, in $R$ we can write (\ref{GR}) as
$G(y,z)-G(\infty,z)\sim\varepsilon^{\nu}\exp\left[  \frac{1}{\varepsilon}%
\Psi(y,z)\right]  \mathbb{K}(y,z).$

\section{The corner layer at \ $(0,1)$}

We determine the constant $\nu$ in (\ref{GRay}) and the function $K_{0}(s)$ in
(\ref{K}) by considering carefully the region where the rays from $(0,1)$
enter the domain $\mathfrak{D,}$ and using asymptotic matching. We introduce
the stretched variable $l,$ the function $F_{l}^{(1)}$and the parameter
$\alpha$ defined by%
\begin{align}
F_{k}(x)  &  =G_{1}(x,k-c+\alpha)=F_{l}^{(1)}(x)\nonumber\\
l  &  =k-c+\alpha,\quad-\infty<l<\infty\label{alpha}\\
\alpha &  =c-\left\lfloor c\right\rfloor ,\quad0<\alpha<1.\nonumber
\end{align}
Note that $\alpha$ is the fractional part of $c$ and $l$ takes on integer
values. Use of (\ref{alpha}) in (\ref{diffeq}) yields the equation%
\begin{equation}
(l-\alpha)\frac{dF_{l}^{(1)}}{dx}=\mu F_{l+1}^{(1)}+\lambda F_{l-1}%
^{(1)}-\left(  \lambda+\mu\right)  F_{l}^{(1)},\quad x>0,\quad l\in%
\mathbb{Z}
. \label{diffcorner}%
\end{equation}

Also, (\ref{FBC}) gives the boundary condition%
\begin{equation}
F_{l}^{(1)}(0)=0,\quad l\geq1 \label{BCcorner}%
\end{equation}
and (\ref{Ginfinity}) implies that $F_{k}(\infty)=F_{l}^{(1)}(\infty)$ with%
\begin{equation}
F_{l}^{(1)}(\infty)=(1-\rho)\rho^{l-\alpha}\exp\left[  \frac{1}{\varepsilon
}\ln\left(  \rho\right)  \right]  . \label{infcorner}%
\end{equation}
For a fixed $l$ and $x\rightarrow\infty,$ we approach the interior of $R,$
where (\ref{GR}) applies. Thus, (\ref{infcorner}) is the asymptotic matching
condition between the corner layer and the solution in $R.$ We shall examine
the matching to $R^{C}$ later.

Since the problem (\ref{diffcorner})-(\ref{infcorner}) is of interest by
itself, we solve it in a slightly more general setting in the next theorem.

\begin{theorem}
\label{Th1}Suppose that the function $\Phi_{l}(x)$ satisfies the equation%
\begin{equation}
(l-\alpha)\frac{d\Phi_{l}}{dx}=A\Phi_{l+1}+C\Phi_{l-1}-B\Phi_{l},\quad
x>0,\quad l\in%
\mathbb{Z}
, \label{Equ}%
\end{equation}
with
\begin{equation}
B>2\sqrt{AC}=\beta>0 \label{beta}%
\end{equation}
and the boundary conditions%
\begin{equation}
\Phi_{l}(0)=0,\quad l\geq1, \label{BCPhi}%
\end{equation}%
\begin{equation}
\Phi_{l}(\infty)=\mu_{\infty}r^{l},\quad l\in%
\mathbb{Z}
, \label{infy}%
\end{equation}
where
\begin{equation}
Ar^{2}-Br+C=0. \label{charac}%
\end{equation}
Then, $\Phi_{l}$ has the integral representation%
\begin{align}
\Phi_{l}(x)  &  =\mu_{\infty}\left(  \frac{C}{A}\right)  ^{\frac{l}{2}}%
\sqrt{\frac{\Delta}{B}}\frac{1}{2\pi\mathrm{i}}\int\limits_{\mathrm{Br}}%
\exp\left(  x\theta+\frac{B-\Delta}{\theta}\right) \label{laplacian}\\
&  \times\frac{1}{\theta}\Gamma\left(  1-\alpha+\frac{B}{\theta}\right)
J_{l-\alpha+\frac{B}{\theta}}\left(  \frac{\beta}{\theta}\right)  \left(
\frac{B-\Delta}{\beta}\frac{B}{\theta}\right)  ^{\alpha-\frac{B}{\theta}%
}d\theta\nonumber
\end{align}
and the spectral representation%
\begin{align}
\Phi_{l}(x)  &  =\mu_{\infty}r^{l}-\mu_{\infty}\left(  \frac{C}{A}\right)
^{\frac{l}{2}}\sqrt{\frac{\Delta}{B}}\sum\limits_{j=0}^{\infty}\frac{\left(
j+1-\alpha\right)  ^{j}}{j!}\label{spectrum}\\
&  \times\left(  \frac{B-\Delta}{\beta}\right)  ^{j+1}\exp\left(  x\theta
_{j}+\frac{B-\Delta}{\theta_{j}}\right)  J_{l-j-1}\left(  \frac{\beta}%
{\theta_{j}}\right)  ,\nonumber
\end{align}
where $\mathrm{Br}$ is a vertical contour in the complex \ $\theta-$plane on
which \ $\operatorname{Re}(\theta)>0,$ $\Gamma\left(  \cdot\right)  $ denotes
the Gamma function, $J_{\nu}(\cdot)$ is the Bessel function of the first
kind,
\[
\Delta=\sqrt{B^{2}-\beta^{2}}=\sqrt{B^{2}-4AC}%
\]
and
\[
\theta_{j}=-\frac{B}{j+1-\alpha},\quad j=0,1,\ldots.
\]

\end{theorem}

\begin{proof}
Equation (\ref{Equ}) admits the separable solutions%
\begin{equation}
\Phi_{l}(x)=e^{\theta x}h_{l}(\theta) \label{Fl}%
\end{equation}
if \ $h_{l}(\theta)$\ satisfies the difference equation%
\[
Ah_{l+1}+Ch_{l-1}=\left[  (l-\alpha)\theta+B\right]  h_{l}.
\]
Setting
\[
\ h_{l}(\theta)=\left(  \frac{C}{A}\right)  ^{\frac{l}{2}}H_{l}(\theta),
\]
we see that \
\begin{equation}
H_{l+1}+H_{l-1}=\frac{2}{\beta}\left[  (l-\alpha)\theta+B\right]  H_{l}.
\label{H}%
\end{equation}
The only solutions to (\ref{H}) which have acceptable behavior as
$l\rightarrow\infty$ are of the form%
\[
H_{l}(\theta)=J_{l-\alpha+\frac{B}{\theta}}\left(  \frac{\beta}{\theta
}\right)
\]
where $J_{\upsilon}(\cdot)$ is the Bessel function. If (\ref{Fl}) is not to
grow as \ $x\rightarrow\infty$, we need $\theta\leq0.$ But except when
$\upsilon$ is an integer, the Bessel function $J_{\upsilon}(\cdot)$ is complex
for negative argument. Therefore, we need%
\[
-\alpha+\frac{B}{\theta}=-1,-2,\ldots
\]
or
\begin{equation}
\theta_{j}=-\frac{B}{j+1-\alpha}<0,\quad j\geq0. \label{thetaj}%
\end{equation}
It follows that the general solution to (\ref{Equ}) takes the form%
\[
\Phi_{l}(x)=\Phi_{l}(\infty)+\left(  \frac{C}{A}\right)  ^{\frac{l}{2}}%
{\displaystyle\sum\limits_{j\geq0}}
a_{j}e^{\theta_{j}x}J_{l-\alpha+\frac{B}{\theta_{j}}}\left(  \frac{\beta
}{\theta_{j}}\right)
\]
or%
\begin{equation}
\Phi_{l}(x)=\mu_{\infty}r^{l}+\left(  \frac{C}{A}\right)  ^{\frac{l}{2}}%
{\displaystyle\sum\limits_{j\geq0}}
a_{j}\exp\left(  -\frac{B}{j+1-\alpha}x\right)  J_{l-1-j}\left[  -\frac{\beta
}{B}(j+1-\alpha)\right]  \label{Fcorner}%
\end{equation}
where the coefficients \ $a_{j}$ \ in the above (spectral) representation
remain to be determined.

Taking the Laplace transform
\[
\widehat{\Phi}_{l}(\theta)=%
{\displaystyle\int\limits_{0}^{\infty}}
e^{-\theta x}\Phi_{l}(x)dx
\]
of (\ref{Fcorner}) we obtain%
\begin{equation}
\widehat{\Phi}_{l}(\theta)=\mu_{\infty}r^{l}\frac{1}{\theta}+\left(  \frac
{C}{A}\right)  ^{\frac{l}{2}}%
{\displaystyle\sum\limits_{j\geq0}}
\frac{a_{j}}{\theta+\frac{B}{j+1-\alpha}}J_{j+1-l}\left[  \frac{\beta}%
{B}(j+1-\alpha)\right]  . \label{lap}%
\end{equation}
Thus, the only singularities of \ $\widehat{\Phi}_{l}(\theta)$ are simple
poles at \ $\theta=0$ and \ $\theta=\theta_{j},\quad\ j\geq0.$ \ It is well
known that the Gamma function \ $\Gamma(z)$ has simple poles at
\ $z=0,-1,-2,\ldots$. Hence, we shall represent $\widehat{\Phi}_{l}(\theta)$
as%
\begin{equation}
\widehat{\Phi}_{l}(\theta)=\left(  \frac{C}{A}\right)  ^{\frac{l}{2}}\frac
{1}{\theta}\Gamma\left(  \frac{B}{\theta}+1-\alpha\right)  J_{l-\alpha
+\frac{B}{\theta}}\left(  \frac{\beta}{\theta}\right)  f(\theta)
\label{laplace1}%
\end{equation}
where \ $f(\theta)$ is chosen such that
\[
\Gamma\left(  \frac{B}{\theta}+1-\alpha\right)  J_{l-\alpha+\frac{B}{\theta}%
}\left(  \frac{\beta}{\theta}\right)  f(\theta)
\]
is analytic for \ $\operatorname{Re}(\theta)>-\frac{B}{1-\alpha}$. Taking the
Laplace transform in (\ref{Equ}), we get the equation%
\[
(l-\alpha)\theta\widehat{\Phi}_{l}(\theta)=A\widehat{\Phi}_{l+1}%
(\theta)+C\widehat{\Phi}_{l-1}(\theta)-B\widehat{\Phi}_{l}(\theta),\ l\geq1,
\]
which is satisfied by (\ref{laplace1}). By the inversion formula for the
Laplace transform, we have%
\[
\Phi_{l}(x)=\left(  \frac{C}{A}\right)  ^{\frac{l}{2}}\frac{1}{2\pi\mathrm{i}}%
{\displaystyle\int\limits_{\mathrm{Br}}}
e^{x\theta}\frac{1}{\theta}\Gamma\left(  \frac{B}{\theta}+1-\alpha\right)
J_{l-\alpha+\frac{B}{\theta}}\left(  \frac{\beta}{\theta}\right)
f(\theta)d\theta,
\]
where $\mathrm{Br}$ is a vertical contour in the complex \ $\theta-$plane on
which \ $\operatorname{Re}(\theta)>0.$

Since the residue of \ $\widehat{\Phi}_{l}(\theta)$ at \ $\theta=0$
\ corresponds to \ $\Phi_{l}(\infty)$, we must have%
\[
\left(  \frac{C}{A}\right)  ^{\frac{l}{2}}\Gamma\left(  \frac{B}{\theta
}+1-\alpha\right)  J_{l-\alpha+\frac{B}{\theta}}\left(  \frac{\beta}{\theta
}\right)  f(\theta)\rightarrow\mu_{\infty}r^{l}%
\]
as \ $\theta\rightarrow0.$ Using the asymptotic formulas ($z\rightarrow
\infty)$%
\begin{equation}
\Gamma\left(  a+bz\right)  \sim\sqrt{2\pi}e^{-bz}\left(  bz\right)
^{a+bz-\frac{1}{2}},\quad b>0 \label{asygamma}%
\end{equation}
and%
\begin{equation}
J_{a+bz}\left(  cz\right)  \sim\frac{1}{\sqrt{2\pi z}\left(  b^{2}%
-c^{2}\right)  ^{\frac{1}{4}}}\exp\left(  z\sqrt{b^{2}-c^{2}}\right)  \left(
\frac{b-\sqrt{b^{2}-c^{2}}}{c}\right)  ^{a+bz},\quad b>c>0, \label{asybessel}%
\end{equation}
in (\ref{laplace1}), we see that%
\[
\Gamma\left(  \frac{B}{\theta}+1-\alpha\right)  J_{l-\alpha+\frac{B}{\theta}%
}\left(  \frac{\beta}{\theta}\right)  \sim e^{\frac{\Delta-B}{\theta}}\left(
\frac{B-\Delta}{\beta}\right)  ^{l}\left(  \frac{B-\Delta}{\beta}\frac
{B}{\theta}\right)  ^{\frac{B}{\theta}-\alpha}\sqrt{\frac{B}{\Delta}}%
,\quad\theta\rightarrow0
\]
or, using (\ref{charac}),%
\[
\Gamma\left(  \frac{B}{\theta}+1-\alpha\right)  J_{l-\alpha+\frac{B}{\theta}%
}\left(  \frac{\beta}{\theta}\right)  \sim e^{\frac{\Delta-B}{\theta}}%
r^{l}\left(  \frac{A}{C}\right)  ^{\frac{l}{2}}\left(  \frac{B-\Delta}{\beta
}\frac{B}{\theta}\right)  ^{\frac{B}{\theta}-\alpha}\sqrt{\frac{B}{\Delta}%
},\quad\theta\rightarrow0.
\]
Therefore, we write%
\begin{equation}
f(\theta)=\mu_{\infty}\sqrt{\frac{\Delta}{B}}\exp\left[  \Upsilon
(\theta)\right]  \tilde{f}(\theta), \label{f}%
\end{equation}
where%
\begin{equation}
\Upsilon(\theta)=\frac{B-\Delta}{\theta}-\left(  \frac{B}{\theta}%
-\alpha\right)  \ln\left(  \frac{B-\Delta}{\beta}\frac{B}{\theta}\right)  ,
\label{Lambda}%
\end{equation}
and\ \ $\tilde{f}(\theta)$ is entire, with \ $\tilde{f}(0)=1.$

By combining the preceding results, we have%
\begin{equation}
\Phi_{l}(x)=\mu_{\infty}\sqrt{\frac{\Delta}{B}}\left(  \frac{C}{A}\right)
^{\frac{l}{2}}\frac{1}{2\pi\mathrm{i}}%
{\displaystyle\int\limits_{\mathrm{Br}}}
e^{x\theta}\frac{1}{\theta}\Gamma\left(  \frac{B}{\theta}+1-\alpha\right)
J_{l-\alpha+\frac{B}{\theta}}\left(  \frac{\beta}{\theta}\right)  \exp\left[
\Upsilon(\theta)\right]  \tilde{f}(\theta)d\theta. \label{Fcorner1}%
\end{equation}
The boundary condition (\ref{BCPhi}) implies that
\[
\underset{\theta\rightarrow\infty}{\lim}\left[  \theta\widehat{\Phi}%
_{l}(\theta)\right]  =0,\quad l\geq1
\]
and using the asymptotic formula%
\[
J_{\nu}(z)\sim\left(  \frac{z}{2}\right)  ^{\nu}\frac{1}{\Gamma\left(
\nu+1\right)  },\quad z\rightarrow0,\quad\nu\neq-1,-2,\ldots
\]
in (\ref{Fcorner1}), we have%
\begin{align*}
&  \frac{1}{\theta}\Gamma\left(  \frac{B}{\theta}+1-\alpha\right)
J_{l-\alpha+\frac{B}{\theta}}\left(  \frac{2\beta}{\theta}\right)  \exp\left[
\Upsilon(\theta)\right] \\
&  \sim\left(  \frac{1}{\theta}\right)  ^{l+1}\frac{\Gamma\left(
1-\alpha\right)  }{\Gamma\left(  l-\alpha+1\right)  }\left(  \frac{\beta}%
{2}\right)  ^{l}\left(  \frac{rB}{C}\right)  ^{\alpha},\quad\theta
\rightarrow\infty.
\end{align*}
Setting $l=1$, we get
\[
\ \tilde{f}(\theta)=o\left(  \theta\right)  ,\quad\theta\rightarrow\infty
\]
and Liouville's theorem forces \ $\tilde{f}(\theta)$ to be a constant. Since
\ $\tilde{f}(0)=1,$ we have $\tilde{f}(\theta)\equiv1.$ Thus, (\ref{Fcorner1})
becomes%
\begin{align}
\Phi_{l}(x)  &  =\mu_{\infty}\sqrt{\frac{\Delta}{B}}\left(  \frac{C}%
{A}\right)  ^{\frac{l}{2}}\frac{1}{2\pi\mathrm{i}}%
{\displaystyle\int\limits_{\mathrm{Br}}}
\left[  e^{x\theta}\frac{1}{\theta}\right. \label{Fcorner2}\\
&  \left.  \times\Gamma\left(  \frac{B}{\theta}+1-\alpha\right)
J_{l-\alpha+\frac{B}{\theta}}\left(  \frac{\beta}{\theta}\right)  \exp\left[
\Upsilon(\theta)\right]  \right]  d\theta\nonumber
\end{align}
and we obtain (\ref{laplacian}).

The coefficients \ $a_{j}$ in the spectral expansion (\ref{Fcorner}) are
determined from (\ref{Fcorner2}) by applying the residue theorem. Noting that%
\[
\operatorname*{Res}\left[  \frac{1}{\theta}\Gamma\left(  \frac{B}{\theta
}+1-\alpha\right)  ,\ \theta=\theta_{j}\right]  =\frac{\left(  -1\right)
^{j}}{\left(  j+1-\alpha\right)  j!},
\]
we obtain%
\begin{equation}
\ a_{j}=\mu_{\infty}\sqrt{\frac{\Delta}{B}}\left(  \frac{C}{A}\right)
^{\frac{l}{2}}e^{x\theta_{j}}\frac{\left(  -1\right)  ^{j}}{\left(
j+1-\alpha\right)  j!}J_{l-\alpha+\frac{B}{\theta_{j}}}\left(  \frac{\beta
}{\theta_{j}}\right)  \exp\left[  \Upsilon(\theta_{j})\right]  ,\quad j\geq0.
\label{aj}%
\end{equation}
Using (\ref{thetaj}) in (\ref{Lambda}), we get%
\[
\exp\left[  \Upsilon(\theta_{j})\right]  =\exp\left(  \frac{B-\Delta}%
{\theta_{j}}\right)  \left(  \frac{B-\Delta}{\beta}\right)  ^{j+1}\left(
\alpha-j-1\right)  ^{j+1}%
\]
and (\ref{spectrum}) follows.
\end{proof}

\vspace{0in}For the problem (\ref{diffcorner})-(\ref{infcorner}), we have%
\begin{equation}
A=\mu,\quad B=\lambda+\mu,\quad C=\lambda,\quad r=\rho,\quad\mu_{\infty
}=\left(  1-\rho\right)  \rho^{c-\alpha},\quad\Delta=\mu-\lambda\label{dictio}%
\end{equation}
and therefore
\begin{align}
F_{l}^{(1)}(x)  &  =\left(  1-\rho\right)  \sqrt{\frac{\mu-\lambda}%
{\mu+\lambda}}\rho^{c-\alpha+\frac{l}{2}}\frac{1}{2\pi\mathrm{i}}%
{\displaystyle\int\limits_{\mathrm{Br}}}
\left[  e^{x\theta}\frac{1}{\theta}\right. \label{cornerint}\\
&  \left.  \times\Gamma\left(  \frac{\lambda+\mu}{\theta}+1-\alpha\right)
J_{l-\alpha+\frac{\lambda+\mu}{\theta}}\left(  \frac{2\sqrt{\mu\lambda}%
}{\theta}\right)  \exp\left[  \Lambda(\theta)\right]  \right]  d\theta
,\nonumber
\end{align}
with%
\[
\Lambda(\theta)=\frac{2\lambda}{\theta}-\left(  \frac{\lambda+\mu}{\theta
}-\alpha\right)  \ln\left(  \sqrt{\rho}\ \frac{\lambda+\mu}{\theta}\right)  .
\]
Also,%
\begin{align}
F_{l}^{(1)}(x)  &  =\left(  1-\rho\right)  \rho^{c-\alpha+\frac{l}{2}}\left[
\rho^{\frac{l}{2}}-\sqrt{\frac{\mu-\lambda}{\mu+\lambda}}\sum\limits_{j=0}%
^{\infty}\frac{\left(  j+1-\alpha\right)  ^{j}}{j!}\right.  \label{cornerspec}%
\\
&  \left.  \times\rho^{\frac{j+1}{2}}\exp\left(  x\vartheta_{j}+\frac
{2\lambda}{\vartheta_{j}}\right)  J_{l-j-1}\left(  \frac{2\sqrt{\mu\lambda}%
}{\vartheta_{j}}\right)  \right]  ,\nonumber
\end{align}
where%
\[
\vartheta_{j}=-\frac{\lambda+\mu}{j+1-\alpha},\quad j\geq0.
\]

This completes the determination of the spectral and integral representations
of $F_{l}^{(1)}(x)$ and hence the leading term for $F_{k}(x)$ in the corner region.

\subsection{\ \ \ Matching the corner and $R^{C}$ regions}

In this section we shall determine the function \ $K_{0}(s)$ in (\ref{K})\ and
the power $\nu$ in (\ref{GRay}). We begin with the following result.

\begin{theorem}
With the same hypothesis and notation as in Theorem \ref{Th1}, let \ $\Omega$
be defined by
\begin{equation}
\Omega=\frac{2\Delta x}{(l-\alpha)^{2}}=O(1). \label{Omega}%
\end{equation}
Then,
\[
\Phi_{l}(x)\sim\mu_{\infty}r^{l}\sqrt{\frac{2B}{3\pi\Delta(l-\alpha)}%
}(1-\Omega)^{-1},\quad x\rightarrow\infty,\quad l\rightarrow\infty,\quad
\Omega\rightarrow1.
\]

\end{theorem}

\begin{proof}
We set%
\[
\theta=\varepsilon\Theta,\quad\eta=(l-\alpha)\theta+B=(z-1)\Theta+B,\quad
\eta,\Theta=O(1),\quad\eta,\Theta>0
\]
and using (\ref{asybessel}), we obtain
\[
J_{\frac{\eta}{\varepsilon\Theta}}\left(  \frac{\beta}{\varepsilon\Theta
}\right)  \sim\frac{\sqrt{\varepsilon\Theta}}{\sqrt{2\pi}\left(  \eta
^{2}-\beta^{2}\right)  ^{\frac{1}{4}}}\exp\left(  \frac{1}{\varepsilon\Theta
}\sqrt{\eta^{2}-\beta^{2}}\right)  \left(  \frac{\eta-\sqrt{\eta^{2}-\beta
^{2}}}{\beta}\right)  ^{\frac{\eta}{\varepsilon\Theta}}%
\]%
\[
=\sqrt{\frac{\varepsilon\Theta}{2\pi p(\eta)}}\exp\left\{  \frac
{1}{\varepsilon\Theta}\left[  p(\eta)+\eta\ln\left(  \frac{\eta-p(\eta)}%
{\beta}\right)  \right]  \right\}  ,\quad\varepsilon\rightarrow0
\]
with%
\begin{equation}
p(\eta)=\sqrt{\eta^{2}-\beta^{2}},\quad p(B)=\Delta. \label{P}%
\end{equation}
Use of (\ref{asygamma}) gives%
\[
\Gamma\left(  \frac{B}{\varepsilon\Theta}+1-\alpha\right)  \sim\sqrt{2\pi}%
\exp\left\{  \frac{B}{\varepsilon\Theta}\left[  \ln\left(  \frac
{B}{\varepsilon\Theta}\right)  -1\right]  \right\}  \left(  \frac
{B}{\varepsilon\Theta}\right)  ^{\frac{1}{2}-\alpha}%
\]
and from (\ref{Lambda}) we have%
\[
\exp\left[  \Upsilon(\varepsilon\Theta)\right]  =\exp\left[  \frac{B-\Delta
}{\varepsilon\Theta}-\frac{B}{\varepsilon\Theta}\ln\left(  \frac{B-\Delta
}{\beta}\frac{B}{\theta}\right)  \right]  \left(  \frac{B-\Delta}{\beta}%
\frac{B}{\varepsilon\Theta}\right)  ^{\alpha}.
\]
Therefore,%
\begin{gather}
J_{\frac{\eta}{\varepsilon\Theta}}\left(  \frac{\beta}{\varepsilon\Theta
}\right)  \Gamma\left(  \frac{B}{\varepsilon\Theta}+1-\alpha\right)
\exp\left[  \Upsilon(\varepsilon\Theta)\right]  \sim\label{asymp1}\\
\sqrt{\frac{B}{p(\eta)}}\left(  \frac{B-\Delta}{\beta}\right)  ^{\alpha}%
\exp\left\{  \frac{1}{\varepsilon\Theta}\left[  p(\eta)+\eta\ln\left(
\frac{\eta-p(\eta)}{\beta}\right)  -\Delta-B\ln\left(  \frac{B-\Delta}{\beta
}\right)  \right]  \right\}  .\nonumber
\end{gather}

Using (\ref{asymp1}) in (\ref{Fcorner2}) yields, in terms of $z$ and $\Omega,$%
\begin{equation}
\Phi_{l}(x)\sim\mu_{\infty}r^{\alpha}\sqrt{\Delta}\left(  \frac{C}{A}\right)
^{\frac{z-1}{2\varepsilon}}\frac{1}{2\pi i}%
{\displaystyle\int\limits_{\mathrm{Br}^{\prime}}}
\frac{1}{\eta-B}\frac{1}{\sqrt{p(\eta)}}\exp\left[  \frac{1}{\varepsilon
}\left(  z-1\right)  g(\eta)\right]  d\eta, \label{asymp2}%
\end{equation}
where%
\begin{equation}
g(\eta)=\frac{\left(  \eta-B\right)  \Omega}{2\Delta}+\frac{1}{\eta-B}\left[
p(\eta)+\eta\ln\left(  \frac{\eta-p(\eta)}{\beta}\right)  -\Delta-B\ln\left(
\frac{B-\Delta}{\beta}\right)  \right]  \label{g}%
\end{equation}
and $\mathrm{Br}^{\prime}$ is a vertical contour in the complex plane with
$\operatorname{Re}(\eta)>B.$ For \ $\varepsilon\rightarrow0,$\ with $\Omega$
\ fixed, we can evaluate (\ref{asymp2}) by the saddle point method to get%
\begin{equation}
\Phi_{l}(x)\sim\mu_{\infty}r^{\alpha}\sqrt{\Delta}\left(  \frac{C}{A}\right)
^{\frac{z-1}{2\varepsilon}}\sqrt{\frac{\varepsilon}{2\pi\left(  z-1\right)  }%
}\frac{1}{\eta^{\ast}-B}\exp\left[  \frac{1}{\varepsilon}\left(  z-1\right)
g(\eta^{\ast})\right]  \frac{1}{\sqrt{p(\eta^{\ast})g^{\prime\prime}%
(\eta^{\ast})}}, \label{saddle}%
\end{equation}
where the saddle point \ $\eta^{\ast}\left(  \Omega\right)  $ \ is defined by
\ $g^{\prime}(\eta^{\ast})=0$. Note that $\eta^{\ast}\left(  \Omega\right)
>B$ for $\Omega<1,$ i.e., the saddle point $\eta^{\ast}\left(  \Omega\right)
$ lies to the right of the pole at $\eta=B$ and the integrand is analytic for
$\operatorname{Re}(\eta)>B$.

Taking the derivative of (\ref{g}), we find that%
\begin{equation}
g^{\prime}(\eta)=\frac{\Omega}{2\Delta}+\frac{1}{\left(  \eta-B\right)  ^{2}%
}\left[  \Delta+B\ln\left(  \frac{B-\Delta}{\eta-p(\eta)}\right)
-p(\eta)\right]  \label{dg}%
\end{equation}
and we observe that \ $g^{\prime}(\eta^{\ast})=0$ \ if \ $\eta^{\ast}=B$ and
$\Omega=1,$ which implies that \ $\eta^{\ast}\left(  1\right)  =B.$ To
determine $\eta^{\ast}$ for $\Omega\sim1,$ we use (\ref{dg}) and an expansion
of the form%
\begin{equation}
\eta^{\ast}\left(  \Omega\right)  \sim B+a_{1}\left(  \Omega-1\right)
+a_{2}\left(  \Omega-1\right)  ^{2}+a_{3}\left(  \Omega-1\right)  ^{3}+\cdots.
\label{etastar}%
\end{equation}
Using (\ref{etastar}) in (\ref{dg}) and expanding the latter in powers of
\ $\Omega-1,$ we find that
\[
a_{1}=-\frac{3\Delta^{2}}{2B},\quad a_{2}=-\frac{27\Delta^{2}}{32B^{3}}\left(
\Delta^{2}-3B^{2}\right)
\]
and
\begin{align*}
g(\eta^{\ast})  &  \sim\ln\left(  \frac{B-\Delta}{\beta}\right)
-\frac{3\Delta}{8B}\left(  \Omega-1\right)  ^{2},\quad g^{\prime\prime}%
(\eta^{\ast})\sim\frac{B}{3\Delta^{3}},\\
\quad\frac{1}{\eta^{\ast}-B}  &  \sim-\frac{2B}{3\Delta^{2}}\left(
\Omega-1\right)  ^{-1},\quad\frac{1}{\sqrt{p(\eta^{\ast})}}\sim\frac{1}%
{\sqrt{\Delta}},
\end{align*}
from which we conclude that
\begin{equation}
\Phi_{l}(x)\sim-\mu_{\infty}r^{l}\sqrt{\frac{2B}{3\pi\Delta(l-\alpha)}}\left(
\Omega-1\right)  ^{-1}. \label{fsaddle1}%
\end{equation}

\end{proof}

\vspace*{0in}

Using (\ref{dictio}) in (\ref{fsaddle1}), we get, for $x,l\rightarrow\infty$
with $\Omega\rightarrow1,$%
\begin{equation}
F_{l}^{(1)}(x)\sim-\left(  1-\rho\right)  \exp\left[  \frac{z\ln\left(
\rho\right)  }{\varepsilon}\right]  \sqrt{\frac{2(\mu+\lambda)\varepsilon
}{3\pi(\mu-\lambda)(z-1)}}\left(  \Omega-1\right)  ^{-1}. \label{fsaddle}%
\end{equation}
This must agree with the behavior of the ray expansion in $R^{C}$ as
$(y,z)\rightarrow(0,1).$ We next evaluate $K$ and $\psi$ in (\ref{GRay}) near
the corner $(0,1).$ From (\ref{z(s,t)}), we have%
\begin{equation}
t\sim\frac{z-1}{\mu-\lambda}-\frac{\left(  z-1\right)  ^{2}\left(  \mu
+\lambda\right)  }{2\left(  \mu-\lambda\right)  ^{3}}s,\quad s\rightarrow
0^{+}. \label{T0p}%
\end{equation}
Using (\ref{T0p}) in (\ref{y(s,t)}), we obtain%
\[
y\sim\frac{\left(  z-1\right)  ^{2}}{2(\mu-\lambda)}-\frac{\left(  z-1\right)
^{3}\left(  \mu+\lambda\right)  }{3(\mu-\lambda)^{3}}s,\quad s\rightarrow
0^{+},
\]
or, using (\ref{Omega}),
\begin{equation}
s\sim-\frac{3}{2}\frac{(\mu-\lambda)^{2}}{\left(  \mu+\lambda\right)  }%
\frac{\Omega-1}{z-1},\quad z\rightarrow1. \label{s}%
\end{equation}
We expand (\ref{psi1}) for small $t$%
\[
\psi(s,t)\sim\ln\left(  \rho\right)  \left[  1+(\mu-\lambda)t\right]  ,\quad
t\rightarrow0,
\]
which taking (\ref{T0p}) into account gives%
\[
\Psi(y,z)\sim\ln\left(  \rho\right)  +(z-1)\ln\left(  \rho\right)  ,\quad
z\rightarrow1
\]
in agreement with the exponential part of (\ref{fsaddle}).

From (\ref{K})-(\ref{J2}), we obtain%
\[
K(s,t)\sim K_{0}(s)\sqrt{\frac{3}{\left(  \mu^{2}-\lambda^{2}\right)  t^{3}}%
},\quad t\rightarrow0,
\]
or, using (\ref{s}) and (\ref{T0p}) in the above,%
\begin{equation}
\mathbb{K}(y,z)\sim K_{0}\left(  -\frac{3}{2}\frac{(\mu-\lambda)^{2}}{\left(
\mu+\lambda\right)  }\frac{\Omega-1}{z-1}\right)  \sqrt{\frac{3(\mu
-\lambda)^{3}}{\left(  \mu^{2}-\lambda^{2}\right)  (z-1)^{3}}},\quad
z\rightarrow1. \label{K1}%
\end{equation}
Matching the algebraic factors in (\ref{fsaddle}) and (\ref{K1}) yields%
\[
\varepsilon^{\nu}K_{0}\left(  -\frac{3}{2}\frac{(\mu-\lambda)^{2}}{\left(
\mu+\lambda\right)  }\frac{\Omega-1}{z-1}\right)  =-\sqrt{\varepsilon}\frac
{2}{3}\frac{\mu+\lambda}{(\mu-\lambda)^{2}}\frac{z-1}{\Omega-1}\sqrt{\frac
{\mu-\lambda}{2\pi}}\left(  1-\rho\right)  ,
\]
which implies that%
\begin{equation}
K_{0}(s)=\sqrt{\frac{\mu-\lambda}{2\pi}}\frac{1-\rho}{s} \label{K0}%
\end{equation}
and \ $\nu=\frac{1}{2}.$ This completes the determination of the asymptotic
solution corresponding to rays from the point \ $(0,1).$ To summarize, we have
established the following.

\begin{result}
The solution of (\ref{eqG}) is asymptotically given by%
\begin{equation}
G(y,z)\sim\sqrt{\varepsilon}\exp\left[  \frac{1}{\varepsilon}\Psi(y,z)\right]
\mathbb{K}(y,z)\text{ \ in }R^{C} \label{G4}%
\end{equation}%
\begin{equation}
G(\infty,z)-G(y,z)\sim-\sqrt{\varepsilon}\exp\left[  \frac{1}{\varepsilon}%
\Psi(y,z)\right]  \mathbb{K}(y,z)\text{ \ in }R \label{G3}%
\end{equation}
with%
\[
G(\infty,z)=\left(  1-\rho\right)  \exp\left[  \frac{1}{\varepsilon}%
z\ln\left(  \rho\right)  \right]  ,
\]%
\[
R=\left\{  0\leq y,\quad0\leq z\leq1\right\}  \cup\left\{  Y_{0}(z)\leq
y,\quad1\leq z\right\}  ,\quad Y_{0}(z)=\frac{\left(  z-1\right)  ^{2}}%
{2(\mu-\lambda)},\quad1\leq z,
\]%
\begin{equation}
\Psi(y,z)=\psi(s,t)=2ys+\left[  \ln\left(  \rho\right)  -st\right]  \left(
z-1\right)  +\ln\left(  \rho\right)  , \label{psi3}%
\end{equation}%
\begin{equation}
\mathbb{K}(y,z)=K(s,t)=\sqrt{\frac{\mu-\lambda}{2\pi\mathbf{J}(s,t)}}%
\frac{1-\rho}{s}, \label{K2}%
\end{equation}
where $(y,z)$ is related to $(s,t)$ by (\ref{y(s,t)}) and (\ref{z(s,t)}), and
$\mathbf{J}(s,t)$ was defined in (\ref{J}). We note that $s<0$ in $R$ so that
the right side of (\ref{G3}) is positive. This gives the leading term for the
probability%
\[
\Pr\left[  X(\infty)>x=\frac{y}{\varepsilon},\quad Z(\infty)=k=\frac
{z}{\varepsilon}\right]
\]
that the buffer exceeds $x=cy.$

In the corner range where (\ref{alpha}) applies, the leading term is given by
(\ref{cornerint}) or (\ref{cornerspec}).
\end{result}

\section{Transition layer}

We shall find a transition layer solution near the curve \ $y=Y_{0}(z)$
defined by (\ref{Y0}) which separates \ $R$ and $R^{C}.$ On this curve
\ $s=0,$ and hence (\ref{K2}) is not valid because \ $\mathbb{K}(y,z)$ \ is
infinite there.

We introduce the new function \ $L_{k}(x)$ defined by%
\[
F_{k}(x)=F_{k}(\infty)L_{k}(x).
\]
Then (\ref{diffeq}) yields for \ $L_{k}(x)$ the equation%
\[
(k-c)L_{k}^{\prime}=\mu L_{k-1}+\lambda L_{k+1}-\left(  \lambda+\mu\right)
L_{k}%
\]
and the boundary condition (\ref{Finf}) becomes%
\begin{equation}
L_{k}(\infty)=1. \label{Linf}%
\end{equation}
In terms of the variables \ $y=\varepsilon^{2}x,\quad z=\varepsilon k,$ the
function $L^{(1)}(y,z)=$ $L_{k}(x)$ satisfies%
\[
\varepsilon(z-1)\frac{\partial L^{(1)}}{\partial y}(y,z)=\mu L^{(1)}%
(y,z-\varepsilon)+\lambda L^{(1)}(y,z+\varepsilon)-\left(  \lambda+\mu\right)
L^{(1)}(y,z)
\]
and therefore, as $\varepsilon\rightarrow0,$%
\[
(z-1)\frac{\partial L^{(1)}}{\partial y}=\left(  \lambda-\mu\right)
\frac{\partial L^{(1)}}{\partial z}+\frac{\lambda+\mu}{2}\frac{\partial
^{2}L^{(1)}}{\partial z^{2}}\varepsilon+O\left(  \frac{\partial^{3}L^{(1)}%
}{\partial z^{3}}\varepsilon^{2}\right)  .
\]
Introducing the stretched variable $\varrho$,\ defined by,%
\begin{equation}
y=Y_{0}(z)+\sqrt{\varepsilon}\varrho=\frac{\left(  z-1\right)  ^{2}}%
{2(\mu-\lambda)}+\sqrt{\varepsilon}\varrho\label{lambda}%
\end{equation}
and the function $L^{(2)}(\varrho,z)=L^{(1)}(y,z),$ we obtain for
\ $L^{(2)}(\varrho,z),$ \ to leading order, the diffusion equation%
\begin{equation}
\left(  \lambda-\mu\right)  \frac{\partial L^{(2)}}{\partial z}+\frac{\left(
\lambda+\mu\right)  }{\left(  \lambda-\mu\right)  ^{2}}\left(  z-1\right)
^{2}\frac{\partial^{2}L^{(2)}}{\partial\varrho^{2}}=0. \label{L}%
\end{equation}
To solve (\ref{L}), we assume that \ $L^{(2)}(\varrho,z)$ \ is a function of
the similarity variable $V=\frac{\varrho}{r(z)},$ and let $\mathfrak{L}%
(V)=L^{(2)}(\varrho,z),$ where \ $r(z)$ \ is not yet determined. From
(\ref{L}) we get%
\begin{equation}
-\left(  \lambda-\mu\right)  r(z)r^{\prime}(z)V\mathfrak{L}^{\prime}%
+\frac{\left(  \lambda+\mu\right)  }{\left(  \lambda-\mu\right)  ^{2}}\left(
z-1\right)  ^{2}\mathfrak{L}^{\prime\prime}=0 \label{L1}%
\end{equation}
and (\ref{Linf}) gives%
\begin{equation}
\mathfrak{L}(\infty)=1. \label{Linf1}%
\end{equation}
We can eliminate $z$ in (\ref{L1}) by choosing \ $r(z)$ to satisfy the
equation%
\begin{equation}
-\left(  \lambda-\mu\right)  r(z)r^{\prime}(z)=\frac{\left(  \lambda
+\mu\right)  }{\left(  \lambda-\mu\right)  ^{2}}\left(  z-1\right)  ^{2}.
\label{mu1}%
\end{equation}
We choose $r(1)=0,$ which is necessary for matching the transition layer with
the corner layer solution (\ref{cornerint}), and solve (\ref{mu1}) to obtain%
\begin{equation}
r(z)=\sqrt{\frac{2}{3}\frac{\mu+\lambda}{\left(  \mu-\lambda\right)  ^{3}}%
}\left(  z-1\right)  ^{\frac{3}{2}}. \label{mu3}%
\end{equation}

Now (\ref{L1}) and (\ref{Linf1}) become%
\[
\mathfrak{L}^{\prime\prime}=-V\mathfrak{L}^{\prime},\quad\mathfrak{L}%
(\infty)=1
\]
with the solution
\[
\mathfrak{L}(V)=\frac{1}{2}\left[  1+\operatorname{erf}\left(  \frac{V}%
{\sqrt{2}}\right)  \right]  =\frac{1}{\sqrt{2\pi}}%
{\displaystyle\int\limits_{-\infty}^{V}}
\exp\left(  -\frac{1}{2}u^{2}\right)  du.
\]
Thus, the transition layer solution for \ $y-Y_{0}(z)=O\left(  \sqrt
{\varepsilon}\right)  $ and $1<z$ is%
\begin{equation}
F_{k}(x)\sim\left(  1-\rho\right)  \rho^{k}\frac{1}{2}\left[
1+\operatorname{erf}\left(  \frac{V}{\sqrt{2}}\right)  \right]  ,
\label{Ftran}%
\end{equation}
with%
\[
V(y,z)=\frac{y-Y_{0}(z)}{\sqrt{\varepsilon}}\sqrt{\frac{3}{2}\frac{\left(
\mu-\lambda\right)  ^{3}}{\mu+\lambda}}\left(  z-1\right)  ^{-\frac{3}{2}}.
\]
We can show that if (\ref{Ftran}) is expanded as $V=\frac{y-Y_{0}%
(z)}{r(z)\sqrt{\varepsilon}}\rightarrow-\infty,$ the transition layer matches
to the ray expansion in $R^{C}$, as given by (\ref{G4}), corresponding to rays
emanating from $(0,1).$

\section{The boundary layer at $z=0$}

The ray expansion in (\ref{G3}) does not satisfy the boundary conditions in
(\ref{BC0}). Thus, we re-examine the problem on the scale $z=O(\varepsilon)$
$\left(  k=O(1)\right)  ,$ with $y>0.$ We consider solutions of (\ref{diffeq})
which have the asymptotic form%
\begin{equation}
F_{k}(x)-F_{k}(\infty)=F_{k}^{(3)}(y)-F_{k}(\infty)\sim\varepsilon^{\nu_{3}%
}\exp\left[  \frac{1}{\varepsilon}\Psi(y,0)\right]  K_{k}^{(3)}(y). \label{F2}%
\end{equation}
Using (\ref{F2}) in (\ref{diffeq}) and expanding in powers of \ $\varepsilon$
\ gives, to leading order,%
\[
0=\lambda K_{k-1}^{(3)}+\mu K_{k+1}^{(3)}+\left[  \Psi_{y}(y,0)-\left(
\lambda+\mu\right)  \right]  K_{k}^{(3)},
\]
or, using (\ref{pq}),%
\begin{equation}
0=\lambda K_{k-1}^{(3)}+\mu K_{k+1}^{(3)}+\left[  S(y,0)-\left(  \lambda
+\mu\right)  \right]  K_{k}^{(3)} \label{eq1}%
\end{equation}
and from (\ref{BC0}) we get%
\begin{equation}
\rho K_{-1}^{(3)}=K_{0}^{(3)}. \label{eq11}%
\end{equation}

From (\ref{z(s,t)}) we have, along $z=0,$%
\begin{equation}
\mu\xi(y)+\lambda\xi^{-1}(y)+S(y,0)-\left(  \lambda+\mu\right)  =0,
\label{xi1}%
\end{equation}
where%
\begin{equation}
\xi(y)=\exp\left[  S(y,0)T(y,0)\right]  \label{xsi}%
\end{equation}
and therefore we can rewrite (\ref{eq1}) as
\begin{equation}
\frac{\rho K_{k-1}^{(3)}+K_{k+1}^{(3)}}{K_{k}^{(3)}}=\rho\xi^{-1}(y)+\xi(y).
\label{eq12}%
\end{equation}
Since $S(y,0)<0$ and $T(y,0)>0,$ we see that $0<\xi(y)<1$ for all $y.$ Using
(\ref{xsi}) in (\ref{y(s,t)}), we have%
\[
\mu\xi(y)-\lambda\xi^{-1}(y)+\lambda-\mu-\left(  \lambda+\mu\right)
\ln\left[  \xi(y)\right]  =S^{2}(y,0)y,
\]
which combined with (\ref{xi1}) gives%
\begin{equation}
\left(  1-\xi^{-1}\right)  \rho-\left(  1-\xi\right)  -\left(  \rho+1\right)
\ln\left(  \xi\right)  =\mu\left[  \left(  1-\xi^{-1}\right)  \rho+\left(
1-\xi\right)  \right]  ^{2}y. \label{eqxi}%
\end{equation}

Solving (\ref{eq12}) subject to (\ref{eq11}), we obtain%
\begin{equation}
K_{k}^{(3)}(y)=\left[  \frac{\xi(y)-1}{1-\rho\xi^{-1}(y)}\xi^{k}(y)+\rho
^{k}\xi^{-k}(y)\right]  \overline{K}(y), \label{K4}%
\end{equation}
with $\overline{K}(y)$ to be determined and hence,
\begin{equation}
F_{k}^{(3)}(y)-F_{k}(\infty)\sim\varepsilon^{\nu_{3}}\exp\left[  \frac
{1}{\varepsilon}\Psi(y,0)\right]  \left[  \frac{\xi(y)-1}{1-\rho\xi^{-1}%
(y)}\xi^{k}(y)+\rho^{k}\xi^{-k}(y)\right]  \overline{K}(y). \label{F3}%
\end{equation}

\vspace*{0in}Setting \ $k=z/\varepsilon,$ $F_{k}^{(3)}(y)-F_{k}(\infty
)=G^{(1)}(y,z)$ in (\ref{F3})\ and letting \ $\varepsilon\rightarrow0,$ \ we
get%
\[
G^{(1)}(y,z)\sim\varepsilon^{\nu_{3}}\exp\left\{  \frac{1}{\varepsilon}%
\Psi(y,0)+\frac{1}{\varepsilon}z\ln\left[  \frac{\rho}{\xi(y)}\right]
\right\}  \overline{K}(y)
\]
or, using (\ref{pq}),
\begin{equation}
G^{(1)}(y,z)\sim\varepsilon^{\nu_{3}}\exp\left[  \frac{1}{\varepsilon}%
\Psi(y,0)+\frac{1}{\varepsilon}z\frac{\partial\Psi}{\partial z}(y,0)\right]
\overline{K}(y). \label{F4}%
\end{equation}

From (\ref{psi2}), (\ref{J}), (\ref{G3}) and (\ref{xsi}) we have%
\begin{equation}
G(y,z)-G(\infty,z)\sim\sqrt{\varepsilon}\exp\left[  \frac{1}{\varepsilon}%
\Psi(y,0)+\frac{1}{\varepsilon}z\frac{\partial\Psi}{\partial z}(y,0)\right]
\sqrt{\frac{\mu-\lambda}{2\pi\mathbf{J}_{0}(y)S(y,0)}}\left(  1-\rho\right)  ,
\label{G2}%
\end{equation}
as $z\rightarrow0,$ with%
\begin{align*}
\Psi(y,0)  &  =2yS(y,0)+\ln\left[  \xi(y)\right]  <0,\quad\mathbf{J}%
_{0}(y)=2\left[  \mu\xi(y)-\frac{\lambda}{\xi(y)}\right]  y-1<0,\quad\\
S(y,0)  &  =\left(  \lambda+\mu\right)  -\mu\xi(y)-\lambda\xi^{-1}(y)<0.
\end{align*}
Matching (\ref{F4}) and (\ref{G2}) we conclude that%
\[
\nu_{3}=\frac{1}{2},\quad\overline{K}(y)=\sqrt{\frac{\mu-\lambda}%
{2\pi\mathbf{J}_{0}(y)S(y,0)}}\left(  1-\rho\right)  .
\]
Therefore, for $k=O(1)$ and $y>0,$ we have%
\begin{gather}
F_{k}^{(3)}(y)-F_{k}(\infty)\sim\sqrt{\varepsilon}\exp\left[  \frac
{1}{\varepsilon}\Psi(y,0)\right]  \left(  1-\rho\right) \label{Fz=0}\\
\times\left[  \frac{\xi(y)-1}{1-\rho\xi^{-1}(y)}\xi^{k}(y)+\rho^{k}\xi
^{-k}(y)\right]  \sqrt{\frac{\mu-\lambda}{2\pi\mathbf{J}_{0}(y)S(y,0)}%
},\nonumber
\end{gather}
where $\xi(y)$ is defined implicitly by (\ref{eqxi}).

\section{The boundary $x=0$}

For $x=0$ and $k\leq\left\lfloor c\right\rfloor ,$ the values of $F_{k}(0)$
can be computed from the ray expansion, since $F_{k}(0)-F_{k}(\infty)\sim
\sqrt{\varepsilon}\mathbb{K}(0,z)\exp\left[  \frac{1}{\varepsilon}%
\Psi(0,z)\right]  $ is well defined. For $x=0$ and $k\geq\left\lfloor
c\right\rfloor +1,$ we have $F_{k}(0)=0$ by (\ref{FBC}). We now examine how
this boundary condition is satisfied by considering the scale $y=O(\varepsilon
)$ and $z>1$. Note that this part of the boundary is in the region $R^{C}.$

\vspace*{0in}From (\ref{z(s,t)}) we have%
\begin{equation}
e^{st}=\frac{\left(  z-1\right)  s+\mu+\lambda+\sqrt{\left[  \left(
z-1\right)  s+\mu+\lambda\right]  ^{2}-4\lambda\mu}}{2\mu},\quad z>1.
\label{st1}%
\end{equation}
Using (\ref{st1}) in (\ref{y(s,t)}) we get%
\begin{equation}
S(y,z)=\frac{z-1}{y}+\frac{1}{z-1}\left\{  \left(  \mu+\lambda\right)
\ln\left[  \frac{\mu}{\left(  z-1\right)  ^{2}}y\right]  +2\lambda\right\}
+O\left(  y\ln^{2}y\right)  , \label{S0}%
\end{equation}
and using (\ref{S0}) in (\ref{st1})
\begin{equation}
T(y,z)=\frac{1}{z-1}\ln\left[  \frac{\left(  z-1\right)  ^{2}}{\mu y}\right]
y+O\left(  y^{2}\ln^{2}y\right)  , \label{T00}%
\end{equation}
for $y\rightarrow0^{+}$ and $z>1.$

Using (\ref{S0}) and (\ref{T00}) in (\ref{psi3}) and (\ref{K2}), we find that%
\begin{gather}
\Psi(y,z)\sim\widetilde{\Psi}(y,z)=\ln\left(  \rho\right)  +\left(
z-1\right)  \left\{  \ln\left[  \frac{\lambda y}{\left(  z-1\right)  ^{2}%
}\right]  +2\right\} \label{psi5}\\
+\frac{1}{z-1}\left\{  \left(  \lambda+\mu\right)  \ln\left[  \frac{\mu
y}{\left(  z-1\right)  ^{2}}\right]  +\lambda-\mu\right\}  y,\quad
y\rightarrow0.\nonumber
\end{gather}
Hence, we shall consider asymptotic solutions of the form%
\begin{equation}
F_{k}(x)\sim\varepsilon^{\sigma}\exp\left[  \frac{1}{\varepsilon}%
\widetilde{\Psi}(\varepsilon x,\varepsilon k)\right]  \widetilde
{K}(u,\varepsilon k), \label{F26}%
\end{equation}
where $u=\varepsilon x,$ $u=O(1)$ and $\sigma,\ \widetilde{K}(u,z)$ are to be
determined. Using (\ref{F26}) in (\ref{eqG}) we get, to leading order,%
\begin{equation}
\left(  z-1\right)  \frac{\partial\widetilde{K}}{\partial z}+u\frac
{\partial\widetilde{K}}{\partial u}+\ \widetilde{K}=0. \label{K4eq}%
\end{equation}
The most general solution to (\ref{K4eq}) is%
\begin{equation}
\ \widetilde{K}(u,z)=\frac{1}{u}\widetilde{k}\left(  \Xi\right)  ,\quad
\Xi=\frac{z-1}{u}. \label{K41}%
\end{equation}
Hence,%
\begin{equation}
F_{k}(x)\sim\widetilde{G}(u,z)=\varepsilon^{\sigma}\exp\left[  \frac
{1}{\varepsilon}\widetilde{\Psi}(\varepsilon u,z)\right]  \frac{1}%
{u}\widetilde{k}\left(  \Xi\right)  . \label{F27}%
\end{equation}
To find $\widetilde{k}\left(  \Xi\right)  $ and \ $\sigma$ \ we will match
(\ref{F27}) with the corner layer solution (\ref{cornerint}).

Recalling that $l-\alpha=\frac{z-1}{\varepsilon}$ and using the asymptotic
formula (\ref{asybessel}) we get, as \ $\varepsilon\rightarrow0$ with $\theta$
fixed%
\begin{equation}
J_{\frac{z-1}{\varepsilon}+\frac{\lambda+\mu}{\theta}}\left(  \frac{2\sqrt
{\mu\lambda}}{\theta}\right)  \sim\sqrt{\frac{\varepsilon}{2\pi\left(
z-1\right)  }}\exp\left\{  \left(  \frac{z-1}{\varepsilon}+\frac{\lambda+\mu
}{\theta}\right)  \ln\left[  \frac{\sqrt{\mu\lambda}e\varepsilon}%
{\theta\left(  z-1\right)  }\right]  -\frac{\left(  \lambda+\mu\right)
}{\theta}\right\}  . \label{J1}%
\end{equation}
Using (\ref{J1}) and writing (\ref{cornerint}) in terms of $u=\varepsilon x$
and $z=1+(l-\alpha)\varepsilon,$ we have%
\begin{gather}
F_{l}^{(1)}(x)\sim\left(  1-\rho\right)  \sqrt{\frac{\mu-\lambda}{\mu+\lambda
}}\exp\left[  \left(  \frac{z+1}{2\varepsilon}-\frac{\alpha}{2}\right)
\ln\left(  \rho\right)  \right]  \sqrt{\frac{\varepsilon}{2\pi\left(
z-1\right)  }}\nonumber\\
\times\frac{1}{2\pi\mathrm{i}}%
{\displaystyle\int\limits_{\mathrm{Br}}}
\left\{  \frac{1}{\theta}\exp\left[  \frac{u\theta}{\varepsilon}+\left(
\frac{z-1}{\varepsilon}+\frac{\lambda+\mu}{\theta}\right)  \ln\left(
\frac{\sqrt{\mu\lambda}e\varepsilon}{\theta\left(  z-1\right)  }\right)
\right]  \right. \label{F17}\\
\left.  \times\Gamma\left(  \frac{\lambda+\mu}{\theta}+1-\alpha\right)
\exp\left[  \frac{\lambda-\mu}{\theta}-\left(  \frac{\lambda+\mu}{\theta
}-\alpha\right)  \ln\left(  \sqrt{\rho}\ \frac{\lambda+\mu}{\theta}\right)
\right]  \right\}  d\theta,\nonumber
\end{gather}
To evaluate (\ref{F17}) asymptotically as $\varepsilon\rightarrow0$ we shall
use the saddle point method. We find that the integrand has a saddle point at
$\theta=\Xi,$ so that%
\begin{gather*}
F_{l}^{(1)}(x)\sim\varepsilon\left(  1-\rho\right)  \sqrt{\frac{\mu-\lambda
}{\mu+\lambda}}\exp\left[  \left(  \frac{z+1}{2\varepsilon}-\frac{\alpha}%
{2}\right)  \ln\left(  \rho\right)  \right]  \frac{1}{2\pi u}\frac{1}{\Xi}\\
\times\Gamma\left(  \frac{\lambda+\mu}{\Xi}+1-\alpha\right)  \exp\left[
\frac{u\Xi}{\varepsilon}+\left(  \frac{z-1}{\varepsilon}+\frac{\lambda+\mu
}{\Xi}\right)  \ln\left(  \frac{\sqrt{\mu\lambda}e\varepsilon}{\Xi\left(
z-1\right)  }\right)  \right] \\
\times\exp\left[  \frac{\lambda-\mu}{\Xi}-\left(  \frac{\lambda+\mu}{\Xi
}-\alpha\right)  \ln\left(  \sqrt{\rho}\ \frac{\lambda+\mu}{\Xi}\right)
\right]  ,
\end{gather*}
or%
\begin{gather}
F_{l}^{(1)}(x)\sim\varepsilon\left(  1-\rho\right)  \sqrt{\frac{1-\rho}%
{1+\rho}}\frac{1}{2\pi u}\frac{1}{\Xi}\Gamma\left(  \frac{\lambda+\mu}{\Xi
}+1-\alpha\right) \nonumber\\
\times\exp\left[  \frac{\ln\left(  \rho\right)  }{\varepsilon}+\frac{u\Xi
}{\varepsilon}\ln\left(  \frac{\lambda e^{2}\varepsilon}{\Xi^{2}u}\right)
+\alpha\ln\left(  \ \frac{\lambda+\mu}{\Xi}\right)  \right] \label{F188}\\
\times\exp\left\{  \frac{\lambda+\mu}{\Xi}\ln\left[  \frac{\varepsilon}{\Xi
u\left(  \rho+1\right)  }\right]  +\frac{2\lambda}{\Xi}\right\}  .\nonumber
\end{gather}

Writing (\ref{F27}) in terms of $\Xi$, we obtain%
\begin{align}
\widetilde{G}(u,u\Xi+1)  &  =\varepsilon^{\sigma}\exp\left[  \frac{\ln\left(
\rho\right)  }{\varepsilon}+\frac{u\Xi}{\varepsilon}\ln\left(  \frac{\lambda
e^{2}\varepsilon}{u\Xi^{2}}\right)  \right] \label{F19}\\
&  \times\exp\left[  \frac{\left(  \lambda+\mu\right)  }{\Xi}\ln\left(
\frac{\mu\varepsilon}{u\Xi^{2}}\right)  +\frac{\left(  \lambda-\mu\right)
}{\Xi}\right]  \frac{1}{u}\widetilde{k}\left(  \Xi\right)  .\nonumber
\end{align}
Matching (\ref{F188}) with (\ref{F19}), we have \
\begin{align*}
\widetilde{k}\left(  \Xi\right)   &  =\left(  1-\rho\right)  \sqrt
{\frac{1-\rho}{1+\rho}}\frac{1}{2\pi}\frac{1}{\Xi}\Gamma\left(  \frac
{\lambda+\mu}{\Xi}+1-\alpha\right) \\
&  \times\exp\left[  \alpha\ln\left(  \ \frac{\lambda+\mu}{\Xi}\right)
+\frac{\lambda+\mu}{\Xi}\ln\left(  \frac{e\Xi}{\lambda+\mu}\right)  \right]
\end{align*}
and \ $\sigma=1.$ Therefore, for $1<z,$
\begin{gather*}
\widetilde{G}(u,u\Xi+1)=\varepsilon\left(  1-\rho\right)  \sqrt{\frac{1-\rho
}{1+\rho}}\frac{1}{2\pi}\frac{1}{\Xi}\frac{1}{u}\Gamma\left(  \frac
{\lambda+\mu}{\Xi}+1-\alpha\right) \\
\times\exp\left[  \frac{\ln\left(  \rho\right)  }{\varepsilon}+\frac{u\Xi
}{\varepsilon}\ln\left(  \frac{\lambda e^{2}\varepsilon}{\Xi^{2}u}\right)
+\alpha\ln\left(  \ \frac{\lambda+\mu}{\Xi}\right)  \right] \\
\times\exp\left\{  \frac{\lambda+\mu}{\Xi}\ln\left[  \frac{\varepsilon}{\Xi
u\left(  \rho+1\right)  }\right]  +\frac{2\lambda}{\Xi}\right\}  ,
\end{gather*}
or%
\begin{gather}
\widetilde{G}(u,z)=\varepsilon\left(  1-\rho\right)  \sqrt{\frac{1-\rho
}{1+\rho}}\frac{1}{2\pi}\frac{1}{z-1}\Gamma\left[  \frac{\left(  \lambda
+\mu\right)  u}{z-1}+1-\alpha\right] \nonumber\\
\times\exp\left\{  \frac{\ln\left(  \rho\right)  }{\varepsilon}+\frac
{z-1}{\varepsilon}\ln\left[  \frac{\lambda e^{2}\varepsilon u}{\left(
z-1\right)  ^{2}}\right]  +\alpha\ln\ \left[  \frac{\left(  \lambda
+\mu\right)  u}{z-1}\right]  \right\} \label{F20}\\
\times\exp\left\{  \frac{\left(  \lambda+\mu\right)  u}{z-1}\ln\left[
\frac{\varepsilon}{\left(  \rho+1\right)  \left(  z-1\right)  }\right]
+\frac{2\lambda u}{z-1}\right\}  ,\nonumber
\end{gather}
Note that from (\ref{F20}) we have $\widetilde{G}(u,\varepsilon k)=O\left(
u^{k-\left\lfloor c\right\rfloor }\right)  ,$ as $u\rightarrow0,$
$k\geq\left\lfloor c\right\rfloor +1.$

\section{The marginal distribution}

We will now find the equilibrium probability that the buffer content exceeds
$x,$%
\begin{equation}
M(x)=\Pr\left[  X(\infty)>x\right]  =1-%
{\displaystyle\sum\limits_{k=0}^{\infty}}
F_{k}(x) \label{M}%
\end{equation}
for various ranges of $x.$

\subsection{Approximation for $x=O(1)$}

In this region we shall use the spectral representation of the corner layer
solution. Using the generating function%
\[%
{\displaystyle\sum\limits_{i=-\infty}^{\infty}}
J_{i}(x)z^{i}=\exp\left[  \frac{x}{2}\left(  z-\frac{1}{z}\right)  \right]  ,
\]
in the form%
\[
\exp\left[  \frac{1-\rho}{\rho+1}(j+1-\alpha)\right]  =\left(  \sqrt{\rho
}\right)  ^{-(j+1)}%
{\displaystyle\sum\limits_{l=-\infty}^{\infty}}
J_{l-(j+1)}\left[  -\frac{2\sqrt{\rho}}{\rho+1}(j+1-\alpha)\right]  \left(
\sqrt{\rho}\right)  ^{l},
\]
we obtain from (\ref{cornerspec})
\begin{align}
M(x)  &  \sim M^{(1)}(x)=\left(  1-\rho\right)  \rho^{c-\alpha+1}\sqrt
{\frac{1-\rho}{1+\rho}}%
{\displaystyle\sum\limits_{j\geq0}}
\frac{\left(  j+1-\alpha\right)  ^{j}}{j!}\rho^{j}\label{M1}\\
&  \times\exp\left[  -\frac{x\left(  \lambda+\mu\right)  }{j+1-\alpha}%
+\frac{1-3\rho}{\rho+1}(j+1-\alpha)\right]  .\nonumber
\end{align}
$\ \ \ \ \ \ $

\subsection{Approximation for $x=O(\varepsilon^{-2})=O\left(  c^{2}\right)  $}

We shall now use the asymptotic solution in the region $R,$ as given by
(\ref{G3}). We have
\begin{equation}
M(x)\sim M^{(2)}(y)=-%
{\displaystyle\sum\limits_{k=0}^{\infty}}
G\left(  y,k\varepsilon\right)  \sim-\frac{1}{\sqrt{\varepsilon}}%
{\displaystyle\int\limits_{0}^{\infty}}
\exp\left[  \frac{1}{\varepsilon}\Psi(y,z)\right]  \mathbb{K(}y,z)dz.
\label{Mar1}%
\end{equation}
To evaluate (\ref{Mar1}) as $\varepsilon\rightarrow0,$ we use the Laplace
method. From (\ref{pq}) we get%
\begin{equation}
\Psi_{z}(y,z)=q=0\quad\Leftrightarrow\quad st=\ln\left(  \rho\right)
\label{st11}%
\end{equation}
and therefore the main contribution to (\ref{Mar1}) comes from $z=1$ and we
obtain%
\[
M^{(2)}(y)\sim-\frac{\sqrt{2\pi}}{\sqrt{-\Psi_{zz}(y,1)}}\exp\left[  \frac
{1}{\varepsilon}\Psi(y,1)\right]  \mathbb{K(}y,1).
\]
Using (\ref{st1}) in (\ref{J})-(\ref{qz}) and (\ref{psi3})-(\ref{K2}), we
obtain%
\begin{align}
\Psi_{zz}(y,1)  &  =\frac{S(y,1)}{\mu-\lambda},\quad\mathbb{K(}y,1)=\frac
{1-\rho}{2S(y,1)}\sqrt{\frac{-S(y,1)}{\pi y}},\label{gauss}\\
\Psi(y,1)  &  =2yS(y,1)+\ln\left(  \rho\right)  ,\quad\nonumber
\end{align}
while (\ref{y(s,t)}) gives%
\begin{equation}
S(y,1)=-\frac{\zeta}{\sqrt{y}}, \label{S(1)}%
\end{equation}
with%
\[
\zeta=\sqrt{2\left(  \lambda-\mu\right)  -\left(  \lambda+\mu\right)
\ln\left(  \rho\right)  }.
\]
Thus,%
\begin{equation}
M^{(2)}(y)\sim\sqrt{\frac{\mu-\lambda}{2}}\frac{1-\rho}{\zeta}\exp\left[
\frac{1}{\varepsilon}(-2\zeta\sqrt{y}+\ln\rho)\right]  . \label{M2}%
\end{equation}
$\ \ \ \ \ $

\section{Summary and discussion}

In most of the domain $\mathfrak{D=}\left\{  (y,z):y,z\geq0\right\}  ,$ the
asymptotic expansion of $F_{k}(x)=G(y,z)$ is given by%

\begin{equation}
G(y,z)\sim\sqrt{\varepsilon}\exp\left[  \frac{1}{\varepsilon}\Psi(y,z)\right]
\mathbb{K}(y,z)\text{ \ in }R^{C} \label{RC1}%
\end{equation}
or%
\begin{equation}
G(\infty,z)-G(y,z)\sim-\sqrt{\varepsilon}\exp\left[  \frac{1}{\varepsilon}%
\Psi(y,z)\right]  \mathbb{K}(y,z)\text{ \ in }R. \label{R1}%
\end{equation}
If we consider the continuous part of the density, given by%
\[
f_{k}(x)=F_{k}^{\prime}(x)=\varepsilon^{2}\frac{\partial G}{\partial
y}(y,z),\quad x>0,
\]
the transition between $R$ and $R^{C}$ disappears, and we have%
\begin{equation}
f_{k}(x)\sim\varepsilon^{\frac{3}{2}}\Psi_{y}(y,z)\exp\left[  \frac
{1}{\varepsilon}\Psi(y,z)\right]  \mathbb{K}(y,z)=\varepsilon^{\frac{3}{2}%
}\exp\left[  \frac{1}{\varepsilon}\psi(s,t)\right]  sK(s,t), \label{density}%
\end{equation}
everywhere in the interior of$\ \mathfrak{D.}$ Note that $\mathbb{K}(y,z)$
becomes infinite along $y=Y_{0}(z)$ (i.e., $s=0),$ but the product $\Psi
_{y}(y,z)\mathbb{K}(y,z)$ remains finite.

The asymptotic expansion of the boundary probabilities $F_{k}(0),$
$k\leq\left\lfloor c\right\rfloor $ can be obtained by setting $y=0$ in
(\ref{R1}). This expression can be used to estimate the difference
\[
F_{k}(\infty)-F_{k}(0)=\Pr\left[  X(\infty)>0,\quad Z(\infty)=k=\frac
{z}{\varepsilon}\right]
\]
which is exponentially small for $\varepsilon\rightarrow0.$ Also, for a fixed
$z\in\lbrack0,1),$ $f_{k}(x)$ is maximal at $x=0$ (see Figure \ref{y=0}). In
other words, if $k<c,$ the buffer will most likely be empty.

For a fixed $z>1,$ $f_{k}(x)$ is peeked along the curve $y=Y_{0}(z)$ (see
Figure \ref{y=y0}). To see this better, we can use (\ref{T0}),
(\ref{inversion}) and (\ref{J}) in (\ref{density}), obtaining%
\[
f_{k}(x)\sim\varepsilon^{\frac{3}{2}}\frac{1-\rho}{\sqrt{2\pi}}\sqrt{\chi
(z)}\exp\left\{  \frac{1}{\varepsilon}\left[  z\ln\left(  \rho\right)
-\chi(z)\left(  y-Y_{0}\right)  ^{2}\right]  \right\}  ,\quad z>1
\]
with%
\[
\chi(z)=\frac{3\left(  \mu-\lambda\right)  ^{3}}{\left(  \mu+\lambda\right)
\left(  z-1\right)  ^{3}},
\]
or equivalently%
\[
f_{k}(x)\sim\left(  1-\rho\right)  \rho^{k}\sqrt{\frac{\chi(\frac{k}{c})}{2\pi
c^{3}}}\exp\left\{  -\frac{\chi(\frac{k}{c})}{c^{3}}\left[  x-c^{2}%
Y_{0}\left(  \frac{k}{c}\right)  \right]  ^{2}\right\}  ,\quad z>1.
\]
This means that given $k>c$ active sources, the most likely value of the
buffer will be
\[
x=c^{2}Y_{0}\left(  \frac{k}{c}\right)  .
\]

\begin{figure}[ptb]
\begin{center}
\rotatebox{270} {\resizebox{10cm}{!}{\includegraphics{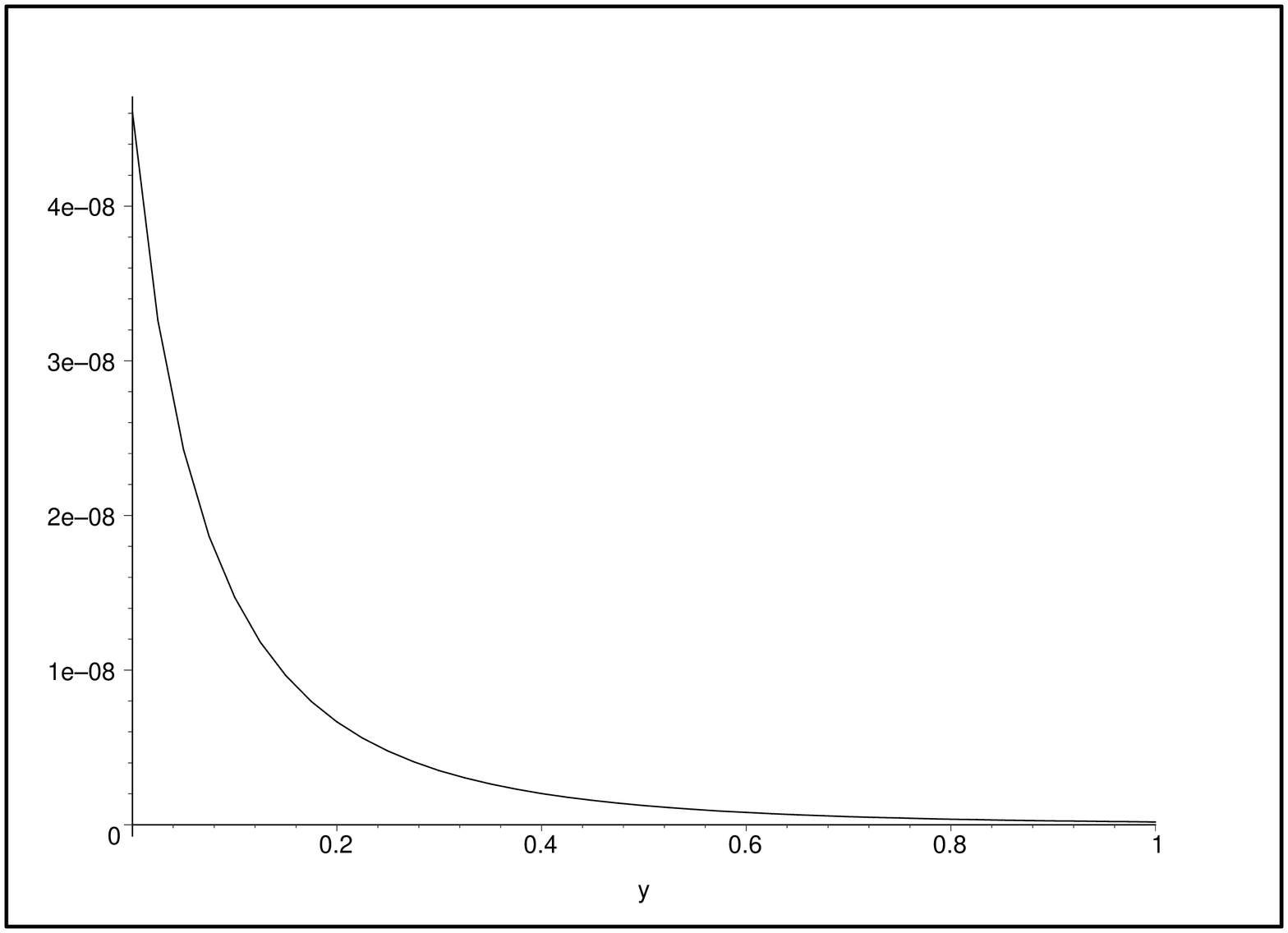}}}
\end{center}
\caption{A plot of $\varepsilon^{\frac{3}{2}}\Psi_{y}(y,z)\exp\left[  \frac
{1}{\varepsilon}\Psi(y,z)\right]  \mathbb{K}(y,z)$, with $\varepsilon
=0.1,\lambda=0.3145, \mu=0.8473$ and $z=0.5$.}%
\label{y=0}%
\end{figure}

\begin{figure}[ptb]
\begin{center}
\rotatebox{270} {\resizebox{10cm}{!}{\includegraphics{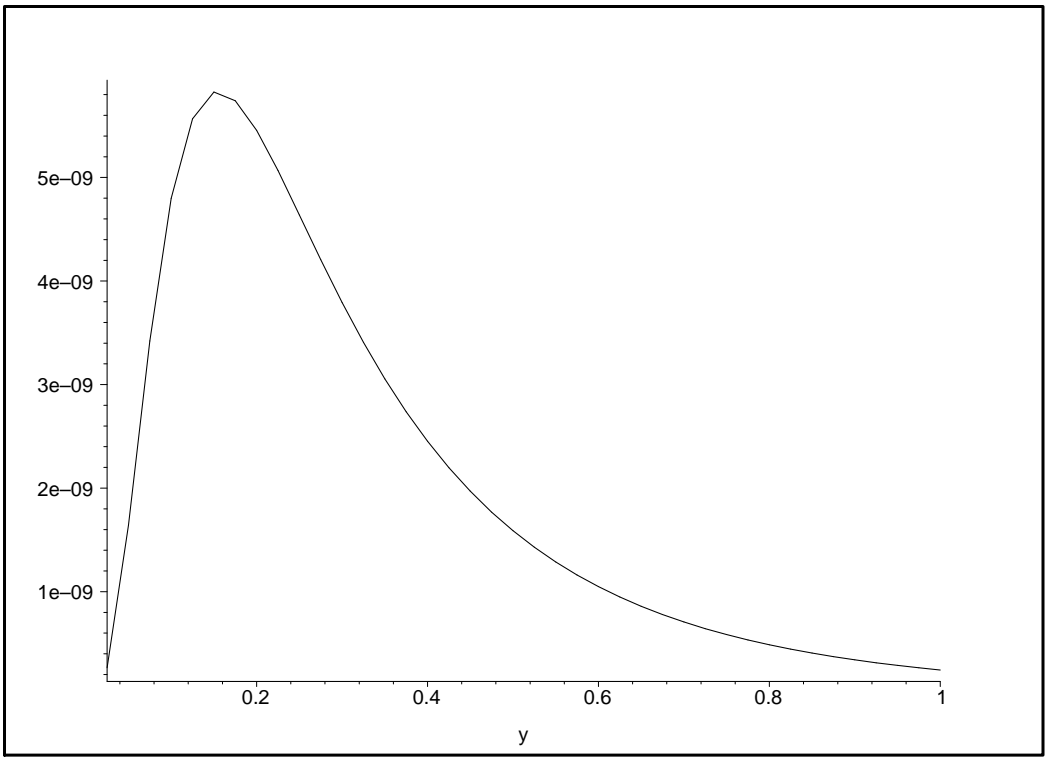}}}
\end{center}
\caption{A plot of $\varepsilon^{\frac{3}{2}}\Psi_{y}(y,z)\exp\left[  \frac
{1}{\varepsilon}\Psi(y,z)\right]  \mathbb{K}(y,z)$, with $\varepsilon
=0.1,\lambda=0.3145, \mu=0.8473$ and $z=1.5$.}%
\label{y=y0}%
\end{figure}

For a fixed $x\geq0,$ $f_{k}(x)$ achieves its maximum around $z=1$ (see Figure
\ref{z=1}). To find an expression for $f_{k}(x)$ when $z$ is close to $1,$ we
use (\ref{gauss}) and obtain, for fixed $y>0,$%
\[
f_{k}(x)\sim\varepsilon^{\frac{3}{2}}\frac{1-\rho}{2}\sqrt{\frac{\zeta}{\pi}%
}y^{-\frac{3}{4}}\exp\left\{  \frac{1}{\varepsilon}\left[  \ln\rho-2\zeta
\sqrt{y}-\frac{\zeta\left(  z-1\right)  ^{2}}{2\sqrt{y}\left(  \mu
-\lambda\right)  }\right]  \right\}  ,
\]
or%
\[
f_{k}(x)\sim\frac{1}{2}\left(  1-\rho\right)  \rho^{c}\sqrt{\frac{\zeta}{\pi}%
}x^{-\frac{3}{4}}\exp\left[  -2\zeta\sqrt{x}-\frac{\zeta\left(  k-c\right)
^{2}}{2\sqrt{x}\left(  \mu-\lambda\right)  }\right]  .
\]

\begin{figure}[ptb]
\begin{center}
\rotatebox{270} {\resizebox{10cm}{!}{\includegraphics{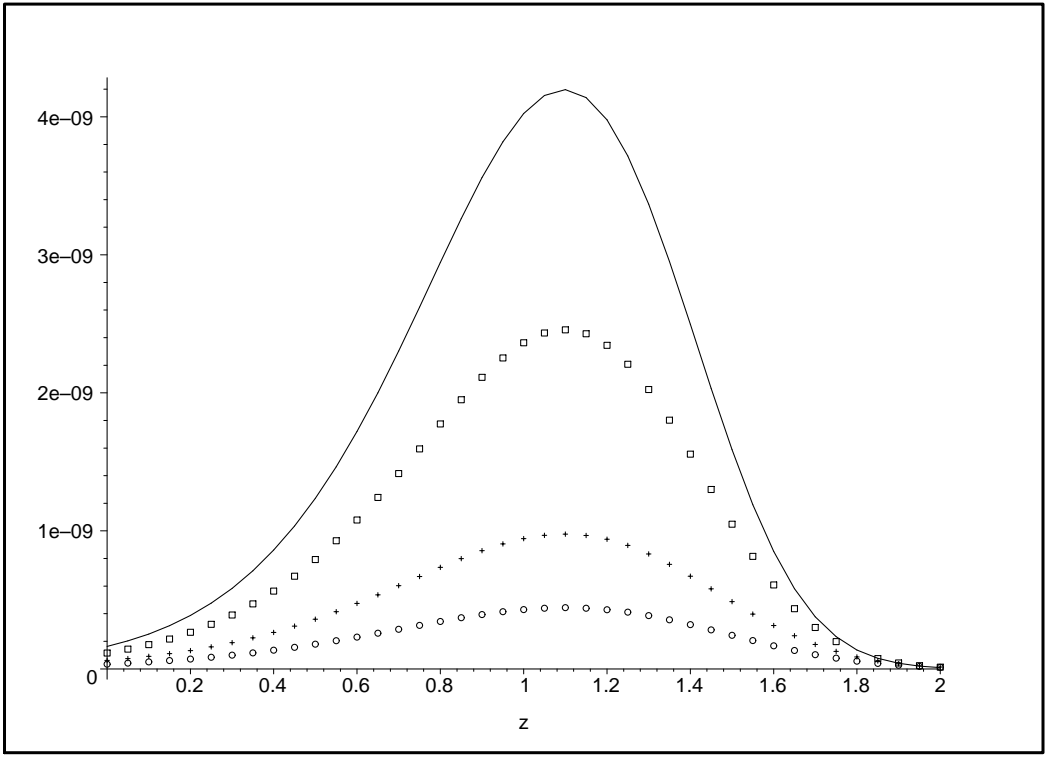}}}
\end{center}
\caption{A plot of $\varepsilon^{\frac{3}{2}}\Psi_{y}(y,z)\exp\left[  \frac
{1}{\varepsilon}\Psi(y,z)\right]  \mathbb{K}(y,z)$, with $\varepsilon
=0.1,\lambda=0.3145, \mu=0.8473$ for $y=0.5$ (solid line), $y=0.6$
($\square),$ $y=0.8$ (+++) and $y=1$ (ooo).}%
\label{z=1}%
\end{figure}

Below we summarize the various boundary, corner and transition layer
corrections to the results in (\ref{G3}) and (\ref{G4}):

\begin{enumerate}
\item $k=l+c-\alpha,$\quad$x=O(1):$%
\begin{align*}
F_{l}^{(1)}(x)  &  =\left(  1-\rho\right)  \sqrt{\frac{\mu-\lambda}%
{\mu+\lambda}}\rho^{c-\alpha+\frac{l}{2}}\\
&  \times\frac{1}{2\pi\mathrm{i}}%
{\displaystyle\int\limits_{\mathrm{Br}}}
e^{x\theta}\frac{1}{\theta}\Gamma\left(  \frac{\lambda+\mu}{\theta}%
+1-\alpha\right)  J_{l-\alpha+\frac{\lambda+\mu}{\theta}}\left(  \frac
{2\sqrt{\mu\lambda}}{\theta}\right)  \exp\left[  \Lambda(\theta)\right]
d\theta,
\end{align*}
where $J_{\cdot}(\cdot)$ denotes the Bessel function, $\Gamma\left(
\cdot\right)  $ the Gamma function, $\mathrm{Br}$ is a vertical contour in the
complex plane with $\operatorname{Re}(s)>0$ and
\[
\alpha=c-\left\lfloor c\right\rfloor \in(0,1),\quad\rho=\frac{\lambda}{\mu
}<1,\quad\Lambda(\theta)=\frac{2\lambda}{\theta}-\left(  \frac{\lambda+\mu
}{\theta}-\alpha\right)  \ln\left[  \sqrt{\rho}\ \frac{\lambda+\mu}{\theta
}\right]  .
\]

\item $y-Y_{0}(z)=O\left(  \sqrt{\varepsilon}\right)  ,\quad1<z:$%
\[
F_{k}(x)\sim\left(  1-\rho\right)  \rho^{k}\frac{1}{2}\left[
1+\operatorname{erf}\left(  \frac{V}{\sqrt{2}}\right)  \right]  ,
\]
with%
\[
V(y,z)=\frac{y-Y_{0}(z)}{\sqrt{\varepsilon}}\sqrt{\frac{3}{2}\frac{\left(
\mu-\lambda\right)  ^{3}}{\mu+\lambda}}\left(  z-1\right)  ^{-\frac{3}{2}%
},\quad Y_{0}(z)=\frac{\left(  z-1\right)  ^{2}}{2\left(  \mu-\lambda\right)
},\quad1<z
\]

\item $k=O(1)$%
\begin{gather*}
F_{k}^{(3)}(y)-F_{k}(\infty)\sim\sqrt{\varepsilon}\exp\left[  \frac
{1}{\varepsilon}\Psi(y,0)\right] \\
\times\left[  \frac{\xi(y)-1}{1-\rho\xi^{-1}(y)}\xi^{k}(y)+\rho^{k}\xi
^{-k}(y)\right]  \sqrt{\frac{\mu-\lambda}{2\pi\mathbf{J}_{0}(y)S(y,0)}}\left(
1-\rho\right)  ,
\end{gather*}%
\begin{gather*}
\Psi(y,0)=2yS(y,0)+\ln\left[  \xi(y)\right]  <0,\quad\mathbf{J}_{0}%
(y)=2\left[  \mu\xi(y)-\frac{\lambda}{\xi(y)}\right]  y-1<0,\\
\quad S(y,0)=\left(  \lambda+\mu\right)  -\mu\xi(y)-\lambda\xi^{-1}%
(y)<0,\quad\\
\left(  1-\xi^{-1}\right)  \rho-\left(  1-\xi\right)  -\left(  \rho+1\right)
\ln\left(  \xi\right)  =\mu\left[  \left(  1-\xi^{-1}\right)  \rho+\left(
1-\xi\right)  \right]  ^{2}y.
\end{gather*}

\item $y=\varepsilon u,\quad u=O(1),\quad1<z$%
\begin{gather*}
F_{k}(x)\sim\varepsilon\left(  1-\rho\right)  \sqrt{\frac{1-\rho}{1+\rho}%
}\frac{1}{2\pi}\frac{1}{z-1}\Gamma\left[  \frac{\left(  \lambda+\mu\right)
u}{z-1}+1-\alpha\right] \\
\times\exp\left\{  \frac{\ln\left(  \rho\right)  }{\varepsilon}+\frac
{z-1}{\varepsilon}\ln\left[  \frac{\lambda e^{2}\varepsilon u}{\left(
z-1\right)  ^{2}}\right]  +\alpha\ln\ \left[  \frac{\left(  \lambda
+\mu\right)  u}{z-1}\right]  \right\} \\
\times\exp\left\{  \frac{\left(  \lambda+\mu\right)  u}{z-1}\ln\left[
\frac{\varepsilon}{\left(  \rho+1\right)  \left(  z-1\right)  }\right]
+\frac{2\lambda u}{z-1}\right\}  .
\end{gather*}

\end{enumerate}

\section*{Acknowledgements}

This work was completed while D. Dominici was visiting Technische
Universit\"{a}t Berlin and supported in part by a Sofja Kovalevskaja Award
from the Humboldt Foundation, provided by Professor Olga Holtz. He wishes to
thank Olga for her generous sponsorship and his colleagues at TU Berlin for
their continuous help.

The work of C. Knessl was partly supported by NSF grant DMS 05-03745.

\end{document}